\documentclass[reqno,a4paper]{amsart}

\usepackage{ifpdf}
\usepackage{graphicx}
\usepackage{epsfig}
\usepackage{color}

\ifpdf
\usepackage[pdftex=true,a4paper]{hyperref}
\hypersetup{
bookmarks=true,
colorlinks=flase,
breaklinks=true,
pdfstartview={FitH},
unicode=true
}
\else
\usepackage[a4paper,hypertex]{hyperref}
\fi

\usepackage{a4wide}
\usepackage[utf8]{inputenc}
\usepackage[english,ngerman]{babel}
\usepackage{amssymb}
\usepackage{amsmath}
\usepackage{mathrsfs}
\usepackage{bbold} % provides serif bblod including the bbold unity
\usepackage{mathptmx} % nice times fonts instead of the wide spread CM 
\usepackage[scaled=.90]{helvet}
\usepackage[sort&compress]{natbib} % powerfull citation package

\newcommand{\titlestring}{Fluctuations of the partition function in
the GREM with external field}
\newcommand{\authorstring}{Anton Bovier and Anton Klimovsky}
\newcommand{\subjectstring}{60K35, 82B44}
\newcommand{\keywordsstring}{generalised random energy
model, spin-glasses, external field, Poisson point processes, extreme values,
probability cascades, weak limit theorems, Gaussian processes}

\ifpdf
\hypersetup{
pdftitle = {\titlestring},
pdfauthor = {\authorstring},
pdfsubject = {\subjectstring},
pdfkeywords = {\keywordsstring},
pdfcreator = {},
pdfproducer = {}
}
\fi

\renewcommand{\P}{\mathbb {P}}
\newcommand{\E}{\mathbb {E}}
\newcommand{\C}{\mathbb {C}}
\newcommand{\R}{\mathbb {R}}
\newcommand{\Z}{\mathbb {Z}}
\newcommand{\N}{\mathbb {N}}

\newcommand{\I}{\mathbb{1}} % bbold unity requires bbold package
\newcommand{\dd}{{\rm d}}
\newcommand{\ee}{{\rm e}}

\DeclareMathOperator{\supp}{supp}
\DeclareMathOperator{\var}{Var}
\DeclareMathOperator{\cov}{Cov}

\DeclareMathOperator{\const}{const}

\DeclareMathOperator{\ch}{ch}

\newcommand{\OO}{\mathcal{O}}

\newcommand{\eps}{\varepsilon}
\newcommand{\sfrac}[2]{\tfrac{#1}{#2}}

\newcommand{\sk}{\mathrm{SK}}
\newcommand{\grem}{\mathrm{GREM}}
\newcommand{\rem}{\mathrm{REM}}
\newcommand{\ppp}{\mathrm{PPP}}
\newcommand{\pointproc}[1]{\mathcal{#1}}

\newtheorem{theorem}{Theorem}[section]

\newtheorem{lemma}{Lemma}[section]

\newtheorem{proposition}{Proposition}[section]

\newtheorem{remark}{Remark}[section] % we want no fancy lines for remarks

\ifx\hypersetup\undefined\else
\fi

\allowdisplaybreaks

\begin{document}

\selectlanguage{english}

\selectlanguage{english}

\begin{center}
\LARGE{\titlestring}
\end{center}
\vskip1cm
\begin{center}
\large{
Anton~Bovier\\
Weierstraß-Institut für Angewandte Analysis und Stochastik\\
Mohrenstraße 39\\
10117 Berlin\\
}
and\\
\large{
Institut für Mathematik, Technische Universität Berlin\\
Straße des 17 Juni 136\\
10623 Berlin\\
e-mail: bovier@wias-berlin.de
}
\end{center}
\vskip1cm
\begin{center}
\large{
Anton~Klimovsky\\
Institut für Mathematik, Technische Universität Berlin\\
Straße des 17 Juni 136\\
10623 Berlin\\
e-mail: klimovsk@math.tu-berlin.de\\
}
\end{center}
\vskip1cm
\begin{center}
Abstract
\end{center}
We study Derrida's generalized random energy model in the presence of 
uniform external field. We compute the fluctuations of the ground state and of
the partition function in the thermodynamic limit for all admissible values of
parameters. We find that the fluctuations are described by a hierarchical
structure which is obtained by a certain coarse-graining of the initial
hierarchical structure of the GREM with external field. We provide an
explicit formula for the free energy of the model. We also derive some large deviation
results providing an expression for the free energy in a class of models with
Gaussian Hamiltonians and external field. Finally, we prove that the
coarse-grained parts of the system emerging in the thermodynamic limit tend to
have a certain optimal magnetization, as prescribed by strength of
external field and by parameters of the GREM.
\vskip0.5cm
\noindent
\textbf{Key words:} \keywordsstring.
\vskip0.5cm
\noindent
\textbf{AMS 2000 Subject Classification:} \subjectstring.
\vskip0.5cm
\noindent

\section{Introduction}
Despite the recent substantial progress due to Guerra \cite{Guerra2003a},
Aizenman, Sims and Starr \cite{AizenmanSimsStarr2003,AizenmanSimsStarr2006}, and
Talagrand \cite{TalagrandParisiFormula2006} in establishing rigorously the
Parisi formula for the free energy of the celebrated Sherrington-Kirkpatrick
(SK) model, understanding of the corresponding limiting Gibbs measure is still very limited.

Due to the above mentioned works, it is now rigorously known that the generalized
random energy model (GREM) introduced by Derrida~\cite{Derrida1985} is closely
related to the SK model at the level of the free energy, see, e.g.,
\cite[Section~11.3]{BovierBook2006}. Recently the first author and Kurkova
\cite{BovierKurkova2004,BovierKurkova2004a,BovierKurkova2003} have performed a
detailed study of the geometry of the Gibbs measure of the GREM. This confirmed
the predicted in the theoretical physics literature hierarchical decomposition of
the Gibbs measure in rigorous terms.

As pointed out in \cite{BovierKurkova2004} (see also
\cite{Limit-Theorems-Sums-Random-Exponentials-2005}), the GREM-like models may
represent an independent interest in various applied contexts, where 
correlated heavy-tailed inputs play an important role, e.g., in risk modeling.

One of the key steps in the results of \cite{BovierKurkova2004} is the
identification of the fluctuations of the GREM partition function in the
thermodynamic limit with Ruelle's probability cascades. In this paper we also
perform this step and study the effect of external field on the fluctuations
(i.e., the weak limit laws) of the partition function of the GREM in the
thermodynamic limit. We find that the main difference introduced by the presence
of external field, comparing to the system without external field, is that the
coarse graining mechanism should be altered. This change reflects the fact that
the coarse-grained parts of the system tend to have a certain optimal
magnetization as prescribed by the strength of external field and by
parameters of the GREM. We use the general line of reasoning suggested in
\cite{BovierKurkova2004}, i.e., we consider the point processes generated by
the scaling limits of the GREM Hamiltonian. We streamline the proof of the
weak convergence of these point processes to the corresponding Poisson point
process by using the Laplace transform.

\subsection*{Organization of the paper}
In the following subsections of the
introduction we define the model of interest and formulate our main results on
the fluctuations of the partition function of the random energy model (REM) and
GREM with external field and also on their limiting free energy
(Theorems~\ref{thm:grem:rem:ground-state-fluctuations},
\ref{thm:grem:limiting-grem-point-process},
\ref{thm:grem:grem:partition-function-fluctuations} and
\ref{thm:grem:grem-free-energy}). Their proofs are provided in the subsequent
sections. Section~\ref{sec:grem:some-preliminary-results} is devoted to the large
deviation results providing an expression for the free energy for a class of
models with Gaussian Hamiltonians and external field
(Theorem~\ref{prp:grem:existence-of-the-partial-thermodynamics-abstract}). In
Section~\ref{sec:grem:the-rem-with-ext-field} we resort to more refined analysis
and perform the calculations of the fluctuations of the ground state and of the
partition function in the REM with external field in the thermodynamic limit.
Section~\ref{sec:grem:grem} contains the proofs of the results on the
fluctuations of the ground state and of the partition function for the GREM with
external field.

\subsection*{Definition of the model}
Derrida's GREM was proposed as a mean-field
spin-glass model with a Gaussian Hamiltonian and hierarchical
correlation structure. In this paper, we consider the GREM with uniform
external (magnetic) field. In contrast to the work of Derrida and
Gardner \cite{DerridaGardnerMagneticProperties1986}, we consider here
the model with the external field which depends linearly on the total
magnetization (i.e., the uniform magnetic field). The authors of
\cite{DerridaGardnerMagneticProperties1986} considered the ``lexicographic''
external field which is particularly well adapted to the natural lexicographic
distance generated by the GREM Hamiltonian.

Given $N \in \N$, consider the standard \emph{discrete hypercube}
$ \Sigma_N \equiv \{-1;1\}^N $. It will play the role of the index set. Define the
(normalized) \emph{lexicographic overlap} between the configurations $\sigma^{(1)},\sigma^{(2)} \in \Sigma_N$ as
\begin{align}
\label{eq:introduction:lexicographic-overlap}
q_\text{L}(\sigma^{(1)},\sigma^{(2)})
\equiv
\begin{cases}
0
,
&
\sigma^{(1)}_1
\neq
\sigma^{(2)}_1
\\
\frac{1}{N}
\max
\left\{
k
\in
[1;N] \cap \N
:
[\sigma^{(1)}]_k
=
[\sigma^{(2)}]_k
\right\}
,
&
\text{otherwise.}
\end{cases}
\end{align}
We equip the
index set with
the \emph{lexicographic distance} defined as
\begin{align*}
\dd_\text{L}(\sigma^{(1)},\sigma^{(2)})
\equiv
1-q_\text{L}(\sigma^{(1)},\sigma^{(2)})
.
\end{align*}
This distance is obviously an \emph{ultrametric}, that is, for all
$
\sigma^{(1)}, \sigma^{(2)}, \sigma^{(3)} \in \Sigma_N
$, 
we have
\begin{align*}
\dd_\text{L}(\sigma^{(1)},\sigma^{(3)})
\leq
\max
\left\{
\dd_\text{L}(\sigma^{(1)},\sigma^{(2)})
,
\dd_\text{L}(\sigma^{(2)},\sigma^{(3)})
\right\}
.
\end{align*}

Throughout the paper, we assume that we are given a large enough probability
space $(\Omega, \mathcal{F}, \P)$ such that all random variables under
consideration are defined on it. Without further notice, we
shall assume that all Gaussian random variables (vectors and processes) are
centered.

Let 
$
\grem_N
\equiv
\{
\grem_N(\sigma)
\}_{
\sigma
\in
\Sigma_N
}
$ 
be the Gaussian random process on the discrete hypercube $\Sigma_N$ with the
covariance of the following form
\begin{align}
\label{eq:introduction:grem-covariance}
\E
\left[
\grem_N(\sigma^{(1)})
\grem_N(\sigma^{(2)})
\right]
=
\varrho
(
q_\text{L}(\sigma^{(1)},\sigma^{(2)})
)
,
\end{align}
where  $\varrho:[0;1]\to[0;1]$ is the non-decreasing right-continuous function
such that $\varrho(0)=0$ and $\varrho(1)=1$.  
Given $h \in \R_+$, consider the Gaussian process 
$
X \equiv X_N \equiv 
\{
X_N(h,\sigma)
\}_{\sigma \in \Sigma_N}
$ 
defined as
\begin{align}
\label{eq:grem:grem-with-ext-field}
X_N(h,\sigma)
\equiv
\grem_N(\sigma)
+
\frac{h}{\sqrt{N}}
\sum_{i=1}^{N}
\sigma_i
,
\quad
\sigma
\in \Sigma_N
.
\end{align}
The second summand in \eqref{eq:grem:grem-with-ext-field} is
called the \emph{external field}. The parameter $h$ represents the
\emph{strength of external field}. Denote the \emph{total
magnetization} by
\begin{align}
\label{eq:grem:total-magnetisation}
m_N(\sigma)
\equiv
\frac{1}{N}
\sum_{i=1}^{N}
\sigma_i
,
\quad
\sigma
\in \Sigma_N
.
\end{align}
The random process \eqref{eq:grem:grem-with-ext-field} induces the \emph{Gibbs
measure} $\mathcal{G}_N(\beta,h) \in \mathcal{M}_1(\Sigma_N)$ in the usual way
\begin{align*}
\mathcal{G}_N(\beta,h)
(
\{\sigma\}
)
\equiv
\frac{1}{Z_N(\beta)}
\exp
\left[
\beta
\sqrt{N}
X_N
\left(
\beta^{-1} h,\sigma
\right)
\right]
,
\end{align*}
where the normalizing constant $Z_N(\beta)$ is called the  \emph{partition
function} $Z_N(\beta,h)$ and is given by the following sum of $2^N$ correlated
exponentials
\begin{align}
\label{eq:grem:partition-function}
Z_N(\beta,h)
\equiv
\sum_{\sigma \in \Sigma_N}
\exp
\left[
\beta
\sqrt{N}
X_N
\left(
h,\sigma
\right)
\right]
.
\end{align}
The real parameter $\beta > 0$ is called the \emph{inverse
temperature}. The important quantities are the \emph{free energy} defined as
\begin{align}
\label{eq:grem:free-energy}
p_N(\beta,h)
\equiv
\frac{1}{N}
\log
Z_N(\beta,h)
,
\end{align}
and the \emph{ground state energy}
\begin{align}
\label{eq:grem:ground-state}
M_N(h)
\equiv
N^{-1/2}
\max_{
\sigma \in \Sigma_N
}
X_N(h,\sigma)
.
\end{align}
In what follows, we shall think of $\beta$ and $h$ as fixed parameters. We
shall occasionally lighten our notation by not indicating the dependence on
these parameters explicitly. 

In this paper we shall mainly be interested in the weak limit theorems (i.e.,
fluctuations) of the partition function \eqref{eq:grem:partition-function} and
of the ground state as $N \uparrow +\infty$. To be precise, the general results
on Gaussian concentration of measure imply that
\eqref{eq:grem:ground-state} and \eqref{eq:grem:free-energy} are self-averaging. By the fluctuations of the ground state, we mean the weak
limiting behavior of the rescaled point process generated by the Gaussian
process \eqref{eq:grem:grem-with-ext-field}. This behavior is studied in
Theorems~\ref{thm:grem:rem:ground-state-fluctuations} and
\ref{thm:grem:limiting-grem-point-process} below. These theorems readily
imply the formulae for the limiting free energy \eqref{eq:grem:free-energy}
and the ground state \eqref{eq:grem:ground-state}. A recent account of the
mathematical results on the GREM without external field and, in particular, on
the behavior of the limiting Gibbs measure can be found in
\cite{BovierKurkova2007}. The GREM with external field was previously
considered by Jana and Rao \cite{JanaRao2006} (see also \cite{Jana2007}),
where its free energy was expressed in terms of a variational problem
induced by an application of Varadhan's lemma. In this work, we apply
very different methods to obtain precise control of the
fluctuations of the partition function for the GREM with external field.
As a simple consequence of these results, we also get a rather
explicit\footnote{In contrast to Jana and Rao \cite[Theorem~5.1]{JanaRao2006}
and Jana \cite[Corollary~4.3.5]{Jana2007}, who stop a the level of variational
problem.} formula for the limiting free energy in the GREM with external field (see Theorem~\ref{thm:grem:grem-free-energy}).

\subsection*{Main results}
In this paper, we shall consider the case of the piece-wise constant
function $\varrho$ with a finite number of jumps. Consider the space of
discrete order parameters
\begin{align}
\label{eq:chap-1:space-of-discrete-order-parameters}
\mathcal{Q}^\prime_n
\equiv
\{
q: [0;1] \to [0;1]
\quad
\mid
\quad
&
q(0)=0
,
q(1)=1
,
\text{$q$ is non-decreasing,}
\nonumber
\\
&
\text{piece-wise constant with $n$ jumps}
\}
.
\end{align}
Recall the function $\varrho$ from \eqref{eq:introduction:grem-covariance}. In
what follows, we shall refer to $\varrho$ as the \emph{discrete order
parameter}. Assume that $\varrho \in
\mathcal{Q}^\prime_n$. In this case, it is possible to construct the process
$\grem_N$ as a finite sum of independent Gaussian processes. Assume that
\begin{align}
\label{eq:framework:grem-order-parameter}
\varrho(x)
=
\sum_{k=1}^n
q_k
\I_{
[x_k; x_{k+1})
}
(x)
,
\end{align}
where 
\begin{align}
\label{eq:chap-1:branching-rates}
0 < & x_1 < \ldots < x_n = 1
,
\\
0 \equiv q_0 < & q_1 < \ldots < q_n = 1
.
\end{align}
Let $\{a_k\}_{k=1}^n \subset \R$ be such that 
$
a^2_k = q_{k}-q_{k-1}
$.
We assume that, for all $k \in
[1;n]\cap\N$, we have $x_k N
\in \N$\footnote{This condition is for notational simplicity. It means that we
actually consider instead of $N$ the increasing sequence $\{N_\alpha\}_{\alpha
\in \N} \subset \N$ such that $N_\alpha \uparrow +\infty$ as $\alpha \uparrow
+\infty$, satisfying $N_\alpha x_k \in \N$, for all $\alpha \in \N$ and all $k \in [1;n] \cap \N$.
}
and also $a_k \neq 0$. Denote $\Delta x_l \equiv x_{l}-x_{l-1}$.

Consider the family of i.i.d standard Gaussian random variables
\begin{align*}
\{
X(\sigma^{(1)},\sigma^{(2)},\ldots,\sigma^{(k)})
\mid
k \in [1;n]\cap\N
,
\sigma^{(1)} 
\in 
\Sigma_{x_1 N}
,
\ldots
,
\sigma^{(k)} 
\in 
\Sigma_{x_k N}
\}
.
\end{align*}
Using these ingredients, for 
$
\sigma
= 
\sigma^{(1)}\shortparallel\sigma^{(2)}\shortparallel\ldots\shortparallel\sigma^{(n)}
\in \Sigma_N
$, we have
\begin{align}
\label{eq:introduction:grem-through-rem-representation}
\grem_N(\sigma)
\sim
\sum_{k=1}^n
a_k
X(\sigma^{(1)},\sigma^{(2)},\ldots,\sigma^{(k)})
.
\end{align}
Equivalence \eqref{eq:introduction:grem-through-rem-representation} is
easily verified by computing the covariance of the right hand side. The computation gives, for 
$
\sigma,\tau \in \Sigma_N
$
\begin{align*}
\cov
\left[
\grem_N(\sigma)
\grem_N(\tau)
\right]
=
q_{
N
q_\text{L}(\sigma,\tau)
}
.
\end{align*}

\subsubsection*{Limiting objects}
We now collect the objects which appear in weak limit theorems for
the GREM partition function and for the ground states. We denote by
$I: [-1;1]\to\R_{+}$ Cram\'er's entropy function, i.e.,
\begin{align}
\label{eq:grem:cramer-entropy}
I(t)
\equiv 
\frac{1}{2}
[(1-t)\log(1-t)+(1+t)\log(1+t)]
.
\end{align}
Define
\begin{align*}
\rho(t)
\equiv
\sqrt{2(\log 2-I(t))} 
,
\end{align*}
\begin{align}
\label{eq:grem:rem:ground-state}
M(h) 
\equiv 
\max_{t\in[-1;1]}
\left(
\rho(t)
+ 
h t
\right)
.
\end{align}
Suppose that the maximum in \eqref{eq:grem:rem:ground-state} is attained at
$t = t_* = t_*(h)$. (The maximum exists and is unique, since $\rho(t)
+ 
h t
$ is strictly concave.)
Consider the following two real sequences
\begin{align}
\label{eq:grem:rem:scaling-a}
A_N(h)
&
\equiv
\left(
\rho(t_*)\sqrt{N}
\right)^{-1}
,
\\
\label{eq:grem:rem:scaling-b}
B_N(h)
&
\equiv
M(h)\sqrt{N}
+
\frac{A_N(h)}{2}
\log
\left(
\frac{
A_N(h)^2
(
I^{\prime\prime}(t_*)
+
h
)
}{
2 \pi (1-t_*^2)
}
\right)
.
\end{align}
Define the \emph{REM scaling function} $u_{N,h}(x):\R\to\R$ as
\begin{align}
\label{eq:grem:rem:energy-scaling}
 u_{N,h}(x) 
\equiv 
A_{N}(h)x + B_{N}(h)
.
\end{align}
Given $f:D \subset \R \to \R_+$, we denote by $ \ppp (f(x) \dd x, x
\in D) $ the Poisson point process with intensity $f$. We start
from a basic limiting object. Assume that the point process $
\mathcal{P}^{(1)} $ on $ \R $ satisfies
\begin{align}
\label{eq:grem:exp-x-ppp}
\mathcal{P}^{(1)}
\sim
\ppp
\left(
\exp(-x) \dd x
,
x \in \R
\right)
,
\end{align}
and is independent of all random
variables around. The point process \eqref{eq:grem:exp-x-ppp} is the limiting
object which appears in the REM.
\begin{theorem}
\label{thm:grem:rem:ground-state-fluctuations}
If $n=1$ (the REM case), then, using the above notations, we have
\begin{align}
\label{grem:rem:energies-weak-convergence}
\sum_{
\sigma\in\Sigma_N
}
\delta_{u^{-1}_{N,h}(X_N(h,\sigma))}
\xrightarrow[N\to\infty]{w}
\mathcal{P}^{(1)}
,
\end{align}
where the convergence is the weak one of the random probability
measures equipped with the vague topology.
\end{theorem}
To formulate the weak limit theorems for the GREM (i.e., for the case $n>1$), we
need a limiting object which is a point process closely
related to the \emph{Ruelle probability cascade}, \citep{Ruelle1987}. Define,
for $
j, k \in [1;n+1] \cap \N
$,
$
j < k
$,
the ``slopes''
corresponding to the function $\varrho$ in
\eqref{eq:introduction:grem-covariance} as
\begin{align*}
\theta_{j,k}
\equiv
\frac{
q_k - q_{j-1}
}{
x_k - x_{j-1}
}
.
\end{align*}
Define also the following $h$-dependent ``modified slopes''
\begin{align*}
\widetilde{\theta}_{j,k}(h)
\equiv
\theta_{j,k}
\rho(t_*(\theta_{j,k}^{-1/2} h))^{-2}
.
\end{align*}
Define the increasing sequence of indices 
$
\{
J_l(h)
\}_{l=0}^{m(h)}
\subset
[0;n+1] \cap \N
$
by the following algorithm. Start from 
$
J_0(h) \equiv 0
$, and define iteratively
\begin{align}
\label{eq:grem:coarse-graining-indeces}
J_l(h)
\equiv
\min
\left\{
J \in [J_{l-1};n+1] \cap \N
:
\widetilde{\theta}_{
J_{l-1}
,
J
}(h)
>
\widetilde{\theta}_{
J+1
,
k
}(h)
,
\text{ for all }
k > J
\right\}
.
\end{align}
Note that $m(h) \leq n$.
The subsequence of indices \eqref{eq:grem:coarse-graining-indeces} induces the
following coarse-graining of the initial GREM
\begin{align}
\label{eq:grem:coarse-grained-parameters}
\bar{q}_l(h)
&
\equiv
q_{
J_l(h)
}
-
q_{
J_{l-1}(h)
}
,
\\
\bar{x}_l(h)
&
\equiv
x_{
J_l(h)
}
-
x_{
J_{l-1}(h)
}
,
\\
\bar{\theta}_l(h)
&
\equiv
\theta_{J_{l-1},J_{l}}
.
\end{align}
The parameters \eqref{eq:grem:coarse-grained-parameters} induce
the new order parameter 
$
\varrho^{(J(h))} \in \mathcal{Q}^\prime_m
$ 
in the usual way
\begin{align*}
\varrho^{(J(h))}(q)
\equiv
\sum_{l=1}^{m(h)}
q_{J_{l}(h)}
\I_{
[
x_{
J_{l}(h)
}
; 
x_{
J_{l+1}(h)
}
)
}
(x)
.
\end{align*}
Define the GREM scaling function 
$
u_{N,\varrho,h}
:
\R 
\to
\R
$
as 
\begin{align*}
u_{N,\varrho,h}(x)
\equiv
\sum_{
l=1
}^{
m(h)
}
\left[
\bar{q}_l(h)^{1/2}
B_{
\bar{x}_l(h) N
}
\left(
\bar{\theta}_l(h)^{-1/2}
h
\right)
\right]
+
N^{-1/2} x
.
\end{align*}
Define the rescaled GREM process as 
\begin{align*}
\overline{
\grem
}_N
(h,\sigma)
\equiv
u_{N,\varrho,h}^{-1}(
\grem_N(h,\sigma)
)
.
\end{align*}
Define the point process of the rescaled GREM energies
$
\mathcal{E}_N
$ 
as 
\begin{align}
\label{eq:grem:pp-of-the-rescaled-grem-energies}
\mathcal{E}_N(h)
\equiv
\sum_{
\sigma \in \Sigma_N
}
\delta_{
\overline{
\grem
}_N
(h,\sigma)
}
.
\end{align}
Consider the following collection of independent point processes (which are
also independent of all random objects introduced above)
\begin{align*}
\{
\mathcal{P}^{(k)}_{
\alpha_1
,
\ldots
,
\alpha_{l-1}
}
\mid
\alpha_1, \ldots, \alpha_{l-1} \in \N
;
l \in [1;m] \cap \N
\}
\end{align*}
such that
\begin{align*}
\mathcal{P}^{(k)}_{
\alpha_1
,
\ldots
,
\alpha_{k-1}
}
\sim
\mathcal{P}^{(1)}
.
\end{align*}
Define the \emph{limiting GREM cascade point process}
$
\mathcal{P}_m
$
on $\R^m$
as follows
\begin{align}
\label{eq:grem:log-of-rpc}
\mathcal{P}_m
\equiv
\sum_{
\alpha
\in 
\N^m
}
\delta_{
(
\mathcal{P}^{(1)}(\alpha_1)
,
\mathcal{P}^{(2)}_{
\alpha_1
}
(\alpha_2)
,
\ldots
,
\mathcal{P}^{(m)}_{
\alpha_1, \alpha_2,\ldots,\alpha_{m-1}
}
(\alpha_m)
)
}
,
\end{align}
Consider the following constants
\begin{align*}
\bar{\gamma}_l(h)
\equiv
\left(
\widetilde{\theta}_{J_{l-1},J_{l}}
\right)^{1/2}
,
\end{align*}
and define the function $E_{h,f}: \R^m \to \R$ as
\begin{align*}
E^{(m)}_{h,\varrho}(e_1,\ldots,e_m)
\equiv
\bar{\gamma}_1(h) e_1
+
\ldots
+
\bar{\gamma}_m(h) e_m
.
\end{align*}
Note that due to \eqref{eq:grem:coarse-graining-indeces}, the constants
$
\{
\bar{\gamma}_l(h)
\}_{l=1}^{m}
$ 
form a decreasing sequence, i.e., for all 
$
l \in [1;m] \cap \N
$, 
we have 
\begin{align}
\label{eq:grem:decreasing-slopes}
\bar{\gamma}_l(h)
> 
\bar{\gamma}_{l+1}(h)
.
\end{align}
The cascade point process \eqref{eq:grem:log-of-rpc} is the limiting object
which describes the fluctuations of the ground state in the GREM.
\begin{theorem}%[the limit of the GREM point process]
\label{thm:grem:limiting-grem-point-process}
We have
\begin{align}
\label{eq:grem:ground-state-fluctuations}
\mathcal{E}_N(h)
\xrightarrow[N \uparrow +\infty]{w}
\int_{
\R^m
}
\delta_{
E^{(m)}_{h,\varrho}(e_1,\ldots,e_m)
}
\mathcal{P}_m(\dd e_1, \ldots, \dd e_m)
\end{align}
and
\begin{align}
\label{eq:grem:limiting-ground-state}
M_N(h)
\xrightarrow[N \uparrow +\infty]{}
\sum_{
l=1
}^{
m(h)
}
\left[
\left(
\bar{q}_l(h)
\bar{x}_l(h)
\right)
^{1/2}
M
\left(
\bar{\theta}_l(h)^{-1/2}
h
\right)
\right]
,
\end{align}
almost surely and in $L^1$.
\end{theorem}
Theorem~\ref{thm:grem:limiting-grem-point-process} allows for complete
characterization of the limiting distribution of the GREM partition function.
To formulate the result, we need the $\beta$-dependent threshold $
l(\beta,h)
\in [0;m] \cap \N
$ 
such that above it all coarse-grained levels $l > l(\beta,h)$ of
the limiting GREM are in the ``high temperature regime''. Below this threshold
the levels $l \leq l(\beta,h)$ are in the ``frozen state''. 
Given 
$
\beta \in \R_+
$,
define 
\begin{align*}
l(\beta,h) 
\equiv
\max
\{
l \in [1;n] \cap \N
:
\beta \bar{\gamma}_l(h) 
>
1
\}
.
\end{align*}
We set $l(\beta,h) \equiv 0$, if $\beta \bar{\gamma}_1(h) \leq 1$.
The following gives full information about the limiting fluctuations of the
partition function at all temperatures.
\begin{theorem}%[the limiting distribution of the GREM partition function ]
\label{thm:grem:grem:partition-function-fluctuations}
We have 
\begin{align}
\label{eq:grem:fluctuations:of-the-partition-function}
&
\exp
\left[
-
\beta
\sqrt{N}
\sum_{l=1}^{
l(\beta,h)
}
\left(
\bar{q}_l(h)^{1/2}
B_{
\bar{x}_l(h) N
}
\left(
\bar{\theta}_l^{-1/2}
h
\right)
\right)
\right]
\nonumber
\\
&
\times
\exp
\left[
-
N
\left(
\log 2
+
\log \ch 
\left(
\beta h 
(
1-x_{
J_{
l(\beta,h)
}
}
)
\right)
+
\frac{1}{2}
\beta^2 
\left(
1-q_{
J_{
l(\beta,h)
}
}
\right) 
\right)
\right]
\ch^{2/3} 
\left(
\beta h 
(
1-x_{
J_{
l(\beta,h)
}
}
)
\right)
\nonumber
\\
&
\times
Z_N(\beta,h)
\xrightarrow[N\uparrow+\infty]{
w
}
K(\beta,h,\varrho)
\int_{
\R^{
l(\beta,h)
}
}
\exp
\left[
\beta
E^{(l(\beta,h))}_{h,\varrho}(e_1,\ldots,e_{l(\beta,h)})
\right]
\mathcal{P}_{
l(\beta,h)
}
(\dd e_1, \ldots, \dd e_{l(\beta,h)})
,
\end{align}
where the constant $K(\beta,h,\varrho)$ depends on $\beta$, $h$ and $\varrho$ only.
Moreover, $K(\beta,h,\varrho) = 1$, if $
\beta \gamma_{l(\beta,h)+1}
<
1
$
and $K(\beta,h,\varrho) \in (0;1)$, if 
$
\beta \gamma_{l(\beta,h)+1}
=
1
$.
\end{theorem}
The above theorem suggests that the increasing sequence of the constants 
$
\{ \beta_l \equiv
\bar{\gamma}^{-1}
\}_{l=1}^{m(h)}
\subset
\R_+
$ 
can be thought as the sequence of the inverse temperatures at which the phase
transitions occur: at $\beta_l$ the corresponding coarse-grained level $l$ of
the GREM with external field ``freezes''.

As a simple consequence of the fluctuation results
of Theorem~\ref{thm:grem:grem:partition-function-fluctuations},
we obtain the following formula for the limiting free energy of the GREM.
\begin{theorem}
\label{thm:grem:grem-free-energy}
We have
\begin{align}
\label{eq:grem:grem-free-energy}
\lim_{N\uparrow+\infty} 
p_N(\beta,h) 
=
&
\beta 
\sum_{l=1}^{l(\beta,h)}
\left[
(
\bar{x}_l
\bar{q}_l
)^{1/2}
\rho(
t_*(\bar{\theta}_l^{-1/2} h)
)
+
h \bar{x}_l t_{*}(\bar{\theta}_l^{-1/2} h)
\right]
\nonumber
\\
&
+
\log 2
+
\log \ch 
\left(
\beta h 
(
1-x_{
J_{
l(\beta,h)
}
}
)
\right)
+
\frac{1}{2}
\beta^2 
\left(
1-q_{
J_{
l(\beta,h)
}
}
\right)
,
\end{align}
almost surely and in $L^1$.
\end{theorem}

\section{Partial partition functions, external
fields and overlaps}
\label{sec:grem:some-preliminary-results}
In this section, we propose a way to compute the free energy of disordered spin
systems with external field using the restricted free energies of systems without
external field. The computation involves a large deviations principle. For gauge
invariant systems, we also show that the partition function of the system with
external field induced by the total magnetization has the same distribution as
the one induced by the overlap with fixed but arbitrary configuration. This
section is based on the ideas of Derrida and Gardner \cite{DerridaGardnerMagneticProperties1986}.

Fix $p \in \N$. Given some finite \emph{interaction
$p$-hypergraph} $(V_N,E^{(p)}_N)$, where $V_N = [1;n] \cap \N$ and $E^{(p)}_N
\subset (V_N)^p$, define the \emph{p-spin
interaction Hamiltonian} as 
\begin{align}
\label{eq:grem:p-spin-interaction-hamiltonian}
X_N(\sigma)
\equiv
\sum_{
i
\in 
E^{(p)}_N
}
J^{(N,p)}_{i}
\sigma_{i_1} 
\sigma_{i_2}
\cdots 
\sigma_{i_p}
,
\quad
\sigma \in \Sigma_N
,
\end{align}
where 
$
J^{(N,p)}
\equiv
\left\{
J^{(N,p)}_{i}
\right\}_{
i \in E^{(p)}_N
}
$ is the collection of random variables having the symmetric joint distribution.
That is, we assume that, for any 
$
\eps^{(1)}
,
\eps^{(2)} 
\in
\{-1;+1\}^{
E^{(p)}_N
}
$, 
and any 
$
t 
\in 
\R^{
E^{(p)}_N
}
$,
\begin{align}
\label{eq:grem:symmetric-joint-distribution}
\E
\left[
\exp
\left(
\text{i} 
\sum_{
r \in E^{(p)}_N
}
t_r \eps^{(1)}_r J^{(N,p)}_{i}
\right)
\right]
=
\E
\left[
\exp
\left(
\text{i} 
\sum_{
r \in E^{(p)}_N
}
t_r \eps^{(2)}_r J^{(N,p)}_{i}
\right)
\right]
,
\end{align}
where $\text{i} \in \C$ denotes the imaginary unit.

A particular important example of
\eqref{eq:grem:p-spin-interaction-hamiltonian} is \emph{Derrida's $p$-spin
Hamiltonian} given by
\begin{align*}
\sk^{(p)}_N(\sigma)
\equiv
N^{-p/2}
\sum_{i_1,\ldots,i_p=1}^N
g_{i_1,\ldots,i_p}
\sigma_{i_1}
\sigma_{i_2}
\cdots
\sigma_{i_p}
,
\end{align*}
where 
$
\{
g_{i_1,\dots, i_p}
\}_{i_1,\dots, i_p = 1}^N
$
is a collection of i.i.d. standard Gaussian random variables. Note that the
condition \eqref{eq:grem:symmetric-joint-distribution} is obviously satisfied.

Given $\rho \in \Sigma_N$, define the corresponding \emph{gauge transformation} $T_\rho : \Sigma_N \to \Sigma_N$ as
\begin{align}
\label{eq:grem:gauge-transformation}
T_\rho(\sigma)_i
=
\rho_i\sigma_i
,
\quad
\sigma \in \Sigma_N
.
\end{align}
Note that the gauge transformation \eqref{eq:grem:gauge-transformation} is
obviously an involution. We say that a $d$-variate random function
$f:\Sigma_N^d \to \R$ is \emph{gauge invariant}, if, for any $\rho \in \Sigma_N$ and any 
$
(\sigma^{(1)},\ldots,\sigma^{(d)})
\in 
\Sigma_N^d
$,
\begin{align*}
f(T_\rho(\sigma^{(1)}),\ldots,T_\rho(\sigma^{(d)}))
\sim
f(\sigma^{(1)},\ldots,\sigma^{(d)})
,
\end{align*}
where $\sim$ denotes equality in distribution.
Define the \emph{overlap} between the configurations 
$
\sigma,\sigma^\prime
\in 
\Sigma_N
$ 
as
\begin{align}
\label{eq:grem:overlap}
R_N(\sigma,\sigma^\prime)
\equiv
\frac{1}{N}
\sum_{i=1}^N
\sigma_i
\sigma^\prime_i
.
\end{align}
Note that the overlap \eqref{eq:grem:overlap} and the lexicographic overlap
\eqref{eq:introduction:lexicographic-overlap} are gauge invariant. 

Given a bounded function $F_N: \Sigma_N \to \R$,
define the \emph{partial partition function}
as
\begin{align}
\label{eq:grem:partial-partition-function-using-overlap}
Z^{(p)}_N(\beta,q,\eps,X_N,F_N)
\equiv
\sum_{
\sigma
:
\vert
F_N(\sigma)
-
q
\vert
\leq
\eps
}
\exp
\left(
\beta
\sqrt{N}
X_N(\sigma)
\right)
.
\end{align}
Denote 
\begin{align}
\label{eq:grem:overlap-range-abstract}
U_N \equiv F_N(\Sigma_N)
,
\quad
U
\equiv
\overline{
\Bigl(
\bigcup_{N=1}^\infty
U_N
\Bigr)
}
.
\end{align}
(The bar in \eqref{eq:grem:overlap-range-abstract} denotes closure in the
Euclidean topology.) Note that for the case $F_N = R_N$ we obviously have
\begin{align*}
U_N
=
\left\{
1-\frac{2 k}{N}
:
k \in [0;N] \cap \Z
\right\}
,
\quad
U = [-1;1]
.
\end{align*}
\begin{proposition}[\cite{DerridaGardnerMagneticProperties1986}]
\label{lem:grem:gauge-invariance-gives-equivalence-between-ext-field-and-overlap}
Assume that $X_N$ is given either by
\eqref{eq:grem:p-spin-interaction-hamiltonian} or $X_N \sim \grem_N$.
Fix
some gauge invariant bivariate function $F_N: \Sigma_N^2 \to \R$,
and
$
q 
\in 
\R
$.

Then, for all $\sigma^\prime, \tau^\prime \in \Sigma_N$, we have
\begin{align}
\label{eq:grem:distributional-invariance-of-the-partition-function}
Z^{(p)}_N(\beta,q,\eps,X_N,F_N(\cdot,\sigma^\prime))
\sim
Z^{(p)}_N(\beta,q,\eps,X_N,F_N(\cdot,\tau^\prime))
.
\end{align}
In particular, the partial partition function
\eqref{eq:grem:partial-partition-function-using-overlap} with 
$
F_N \equiv R_N(\cdot,\sigma^\prime)
$ 
has the same
distribution as the partial partition function which corresponds to fixing the total magnetization
\eqref{eq:grem:total-magnetisation}, i.e.,
\begin{align*}
Z^{(p)}_N(\beta,q,\eps,X_N,R_N(\cdot,\sigma^\prime))
\sim
Z^{(p)}_N(\beta,m,\eps,)
\equiv
\sum_{
\sigma
:
\vert
m(\sigma)
-
q
\vert
<
\eps
}
\exp
\left(
\beta
\sqrt{N}
X_N(\sigma)
\right)
.
\end{align*}
\end{proposition}
\begin{remark}
The proposition obviously remains valid for the Hamiltonians $X_N$ given by
the linear combinations of the p-spin Hamiltonians
\eqref{eq:grem:p-spin-interaction-hamiltonian} with varying $p \in \N$.
\end{remark}
\begin{proof}
\begin{enumerate}
\item 
If $X_N$ is defined by \eqref{eq:grem:p-spin-interaction-hamiltonian}, then
\eqref{eq:grem:distributional-invariance-of-the-partition-function} follows due
to the gauge invariance of \eqref{eq:grem:p-spin-interaction-hamiltonian} and
$F_N$. Indeed, there exists $\rho \in \Sigma_N$ such that $\sigma^\prime =
T_\rho(\tau^\prime)$. Define
\begin{align*}
J^{(N,p,\rho)}_{i}
\equiv
J^{(N,p)}_{i}
\rho_{i_1} \cdots \rho_{i_p}
.
\end{align*}
Due to the symmetry of the joint distribution of $J^{(N,p)}$, we have
\begin{align*}
\{
X_N(\sigma)
\}_{
\sigma \in \Sigma_N
}
\sim
\{
X_N(\sigma)\vert_{
J^{(N,p)}
=
J^{(N,p,\rho)}
}
\}_{
\sigma \in \Sigma_N
}
\end{align*}
which implies
\eqref{eq:grem:distributional-invariance-of-the-partition-function}.
\item
If $X_N = \grem_N$, then, since $X_N$ is a Gaussian process, to prove the
equality in distribution, it is enough to check that the covariance of $X_N$
is gauge symmetric. Equivalence 
\eqref{eq:grem:distributional-invariance-of-the-partition-function} follows,
due to \eqref{eq:introduction:grem-covariance} and the fact that the lexicographic overlap \eqref{eq:introduction:lexicographic-overlap} is gauge invariant.
\end{enumerate}
\end{proof}
The
partial partition function \eqref{eq:grem:partial-partition-function-using-overlap} induces the
\emph{restricted free energy} in the usual way:
\begin{align}
\label{eq:grem:restricted-free-energy}
p^{(p)}_N(\beta,q,\eps,X_N,F_N)
\equiv
\frac{1}{N}
\log
Z^{(p)}_N(\beta,q,\eps,X_N,F_N)
.
\end{align}
Given 
$
\sigma^{(1)}, \sigma^{(2)} \in \Sigma_N
$, 
let 
\begin{align*}
C_N(\sigma^{(1)},\sigma^{(2)})
\equiv
\E
\left[
X_N(\sigma^{(1)})
X_N(\sigma^{(2)})
\right]
,
\quad
\widetilde{C}_N(\sigma^{(1)})
\equiv
C_N(\sigma^{(1)},\sigma^{(1)})
.
\end{align*}
Define 
\begin{align*}
V_N 
\equiv 
\{
C_N(\sigma,\sigma)
:
\sigma \in \Sigma_N
\}
,
\quad
V
\equiv
\overline{
\Bigl(
\bigcup_{N=1}^\infty
V_N
\Bigr)
}
.
\end{align*}
The following result establishes a large deviations type relation between
the partial free energy and the full one.
\begin{theorem}
\label{prp:grem:existence-of-the-partial-thermodynamics-abstract}
Assume $X_N = \{ X_N(\sigma)\}_{\sigma \in \Sigma_N}$ is a centered Gaussian 
process and $F_N: \Sigma_N \to \R$ are such that, for all $N, M \in \N$, all 
$
\sigma^{(1)},\sigma^{(2)}
\in 
\Sigma_{N}
$,
and all
$
\tau^{(1)},\tau^{(2)}
\in 
\Sigma_{M}
$,
\begin{align}
\label{eq:grem:convexity-of-the-covariance}
C_{N+M}
(
\sigma^{(1)} \shortparallel \tau^{(1)}
,
\sigma^{(2)} \shortparallel \tau^{(2)}
)
&
\leq
\frac{N}{N+M}
C_{N}
(
\sigma^{(1)}
,
\sigma^{(2)}
)
+
\frac{M}{N+M}
C_{M}
(
\tau^{(1)}
,
\tau^{(2)}
)
,
\\
\label{eq:grem:convexity-of-the-field}
F_{N+M}(\sigma^{(1)} \shortparallel \tau^{(1)})
&
\leq
\frac{N}{N+M}
F_N(\sigma^{(1)})
+
\frac{M}{N+M}
F_M(\tau^{(1)})
.
\end{align}
Assume that $C_N$ and $F_N$ are bounded uniformly in $N$.

Then
\begin{enumerate}
\item 
The following holds
\begin{align}
\label{eq:grem:existence-of-thermodynamics-abstract}
p_N(\beta,X_N,F_N)
&
\equiv
\frac{1}{N}
\log
\sum_{
\sigma \in \Sigma_N
}
\exp
\left(
\beta
\sqrt{N}
X_N(\sigma)
+
N
F_N(\sigma)
\right)
\nonumber
\\
&
\xrightarrow[N \uparrow +\infty]{}
p(\beta,X,F)
,
\quad
\text{almost surely and in $L^1$.}
\end{align}
\item
The limiting free energy $p(\beta,X,F)$ is almost surely deterministic.
\item
We have
\begin{align}
\label{eq:grem:localisation-convergence}
\lim_{\eps \downarrow +0}
\lim_{N \uparrow +\infty}
p^{(p)}_N(\beta,q,\eps,X_N,F_N)
&
\equiv
p^{(p)}(\beta,q,X,F)
\nonumber
\\
&
=
\sup_{
v \in V
}
\inf_{
\lambda \in \R
,
\gamma \in \R
}
\left(
-\lambda q
-\gamma v
+
p(\beta,X,\lambda F + \gamma \widetilde{C})
\right)
,
\nonumber
\\
&
\text{almost surely and in $L^1$.}
\end{align}
\item
Finally,
\begin{align}
\label{eq:grem:global-free-energy-through-local-free-energy-abstract-ldp}
p(\beta,X,F)
=
\sup_{
q \in U
}
\left(
p^{(p)}(\beta,q,X,F)
+
q
\right)
.
\end{align}
\end{enumerate}
\end{theorem}
\begin{remark}
\mbox{}
\begin{enumerate}
\item
If there exists $\{\const_N \in \R_+ \}_{N=1}^\infty$ such that, for all
$\sigma \in \Sigma_N$,
\begin{align}
\label{eq:grem:constant-variance-assumption}
\widetilde{C}_N(\sigma) 
=
\const_N
,
\end{align}
then \eqref{eq:grem:localisation-convergence} simplifies to
\begin{align}
\label{eq:grem:localisation-convergence-constant-variance}
p^{(p)}(\beta,q,X,F)
=
\inf_{
\lambda \in \R
}
\left(
-\lambda q
+
p(\beta,X,\lambda F)
\right)
,
\quad
\text{almost surely and in $L^1$.}
\end{align}
\item 
Inequality \eqref{eq:grem:convexity-of-the-field} can alternatively be
substituted by the assumption (see
\cite[Theorem~1]{Guerra-Toninelli-Generalized-SK-2003}) that $F_N(\sigma) =
f(S_N(\sigma))$, where $f: \R \to \R$,  $f \in C^1(\R)$, and $S_N: \Sigma_N
\to \R$ is the bounded function such that, for all $\sigma \in \Sigma_N$, $\tau
\in \Sigma_M$,
\begin{align*}
S_{N+M} 
(
\sigma \shortparallel \tau
)
=
\frac{N}{N+M} 
S_N
(
\sigma
)
+
\frac{M}{N+M}
S_M
(
\tau
)
.
\end{align*}
\item
It is easy to check that the assumptions of
Proposition~\ref{prp:grem:existence-of-the-partial-thermodynamics-abstract} are fulfilled, e.g., for 
\begin{align*}
X_N 
\equiv 
c_1 \grem_N + c_2 \sk^{(p)}_N
, 
\end{align*}
and 
\begin{align*}
F_N(\cdot) 
\equiv 
f_1(R_N(\cdot,\sigma^{(N)})) 
+
f_2(q_\text{L}(\cdot,\sigma^{(N)}))
, 
\end{align*}
where 
$
\sigma^{(N)} \in \Sigma_N
$, 
$
c_1, c_2 \in \R
$, 
and 
$
f_1, f_2: \R
\to \R
$, 
such that 
$
f_1 \in C^1(\R)
$, 
$f_2$ is convex. Note that in this case, due to
Proposition~\ref{lem:grem:gauge-invariance-gives-equivalence-between-ext-field-and-overlap},
the free energies \eqref{eq:grem:existence-of-thermodynamics-abstract} and
\eqref{eq:grem:localisation-convergence} does not depend on the choice of the
sequence
 $
\{
\sigma^{(N)}
\}_{N=1}^\infty
 \subset \Sigma_N
$.
\end{enumerate}
\end{remark}
\begin{proof}
Similarly to \cite[Theorem~1]{Contucci2003} and
\cite[Theorem~1]{Guerra-Toninelli-Generalized-SK-2003} we obtain
\eqref{eq:grem:existence-of-thermodynamics-abstract}. Then
\eqref{eq:grem:existence-of-thermodynamics-abstract} implies that
\begin{align*}
p(\beta,X_N,\lambda F_N + \gamma \widetilde{C}_N)
\xrightarrow[N \uparrow +\infty]{}
p(\beta,X,\lambda F + \gamma \widetilde{C})
,
\quad
\text{almost surely and in $L^1$.}
\end{align*}
Hence, we can apply the quenched
large deviation results \cite[Theorems~3.1 and
3.2]{BovierKlimovskyAS2Guerra2008} which readily yield 
\eqref{eq:grem:global-free-energy-through-local-free-energy-abstract-ldp}
and \eqref{eq:grem:localisation-convergence} (or
\eqref{eq:grem:localisation-convergence-constant-variance}, in the case of
\eqref{eq:grem:constant-variance-assumption}).
\end{proof}
\begin{remark}
Derrida and
Gardner \cite{DerridaGardnerMagneticProperties1986} sketched a calculation
of the free energy defined
in \eqref{eq:grem:existence-of-thermodynamics-abstract} in the following case
\begin{align}
\label{eq:grem:assumptions-of-derrida-gardner}
\text{$F_N = q_\text{L}$ and $X_N = \grem_N$}
.
\end{align}
This case is easier than the case \eqref{eq:grem:free-energy} we are
considering here, since both $q_\text{L}$ and $\grem_N$ have lexicographic nature, cf.
\eqref{eq:introduction:grem-covariance} and
\eqref{eq:introduction:lexicographic-overlap}.
\end{remark}

\section{The REM with external field revisited}
\label{sec:grem:the-rem-with-ext-field}
In this section, we recall some known results on the limiting free energy of
the REM with external field. However, we give some new proofs of these
results which illustrate the approach of
Section~\ref{sec:grem:some-preliminary-results}. Moreover, we prove the weak
limit theorem for the ground state and for the partition function of the REM
with external field.

Recall that the REM corresponds to the case $n=1$ in
\eqref{eq:introduction:grem-through-rem-representation}. This implies that the
process $X$ is simply a family of $2^N$ i.i.d. standard Gaussian random
variables. To emphasize this situation we shall write $\text{REM}(\sigma)$ instead of $\text{GREM}(\sigma)$.

\subsection{Free energy and ground state}

Let us start by recalling the following well-known result on the
REM.
\begin{theorem}[\cite{Derrida1980,Eisele1983,OlivieriPicco1984}]
Assume that $n=1$ and let $p(\beta,h)$ be given by \eqref{eq:grem:free-energy}.
The following assertions hold
\begin{enumerate}
\item
We have
\begin{align}
\label{rem-limit}
    \lim_{N\to\infty} p_N(\beta,0)
    =
    \begin{cases}
        \frac{\beta^2}{2}+\log 2,& \beta \leq \sqrt{2\log 2}
        \\
        \beta \sqrt{2\log 2},& \beta \geq \sqrt{2\log 2}
    \end{cases}
,
\quad
\text{almost surely and in $L^1$.}
\end{align}
\item
For all $\beta \geq \sqrt{2\log 2}$ and $ N \in \N$, we have
\begin{align}
\label{rem-low-temp-bound}
0 
\leq 
\E [p_N(\beta,0)] 
\leq 
\beta \sqrt{2\log 2}
.
\end{align}
\end{enumerate}
\end{theorem}
See, e.g., \cite[Theorem~9.1.2]{BovierBook2006} for a short proof.
Given $k \in [0;N] \cap \N$, define the set of configurations having a given
magnetization
\begin{align}
\label{eq:grem:configurations-at-given-magnetisation}
\Sigma_{N,k}\equiv\{\sigma \in \Sigma_N : \sum_{i=1}^{N} \sigma_i = N-2k \}
.
\end{align}
\begin{lemma}
\label{prp:grem:binom-asympt}
Set 
$
t_{k,N} \equiv \frac{N-2k}{N}
$.
Given any $\eps > 0$, uniformly in $k \in [0;N] \cap \N$  such that 
\begin{align*}
t_{k,N}
\in
[-1+\eps; 1-\eps]
,
\end{align*}
we have the following asymptotics
\begin{align}
\label{eq:grem:binom-asymptotics}
\binom{N}{k} 
\underset{N \uparrow +\infty}{=}
\sqrt{\frac{2}{\pi}}\frac{2^N\ee^{-NI(t_{k,N})}}{\sqrt{N(1-t_{k,N}^2)}}
\left(
1
+
\frac{1}{N}
\left(
\frac{1}{12}+\frac{1}{3(1-t_{k,N}^2)}
\right)
+
\OO
\left(\frac{1}{N^2}\right)
\right)
.
\end{align}
\end{lemma}
\begin{proof}
A standard exercise on Stirling's formula.
\end{proof}
\begin{theorem}[\cite{Dorlas2001}]
\label{thm:grem:rem-ext-field}
Assume that $n=1$ (the REM case) and let $p(\beta,h)$ be given by
\eqref{eq:grem:free-energy}. We have
\begin{align}
\label{eq:grem:rem:ext-field-free-energy}
p(\beta,h)
&
\equiv 
\lim_{N\to\infty} p_N(\beta,h)
\nonumber
\\
&=
\begin{cases}
\log 2
+
\log \ch \beta h
+
\frac{\beta^2}{2} 
,
& 
\beta 
\leq 
\sqrt{2(\log 2 - I(t_{*}))} 
\equiv 
\beta_0
\\
\beta(\sqrt{2(\log 2 - I(t_{*}))}+ht_{*}),
& 
\beta \geq \sqrt{2(\log 2 - I(t_{*}))}
\end{cases}
,
\quad
\text{almost surely and in $L^1$,}
\end{align}

and $t_{*}\in(-1;1)$ is a unique maximizer of the following concave function
\begin{align*}
(-1;1) \ni t \mapsto ht+\sqrt{2(\log 2 - I(t))}
.
\end{align*}
\end{theorem}
\begin{proof}
For the sake of completeness, we give a short proof based on (the ideas of)
Theorem~\ref{prp:grem:existence-of-the-partial-thermodynamics-abstract}.
Put
\begin{align*}
    M_{k,N} \equiv
    \begin{cases}
        \lfloor\log_2 \binom{N}{k}\rfloor,& k \in [1;N-1] \cap \N
        \\
        1,& k \in \{0,N\}
    \end{cases}
,
\end{align*}
where $\lfloor x \rfloor$ denotes the largest integer smaller than $x$.
Consider the free energy (cf.
\eqref{eq:grem:restricted-free-energy}) of the REM of volume $M_{k,N}$
\begin{align*}
p_{k,N}(\beta) 
\equiv 
\frac{1}{M_{k,N}}
\log
\sum_{\sigma\in\Sigma_N^k} 
\exp
\Bigl(
\beta 
M_{k,N}^{1/2}
\rem(\sigma)
\Bigr)
,
\end{align*}
where 
$
\rem \equiv 
\{
\rem(\sigma)
\}_{
\sigma \in \Sigma_N
}$ 
is the family of standard i.i.d. Gaussian random variables.
Let
\begin{align}
\label{eq:grem:rem:restricted-free-energy}
\widetilde{p}_{k, N}(\beta) 
\equiv
\frac{M_{k,N}}{N} 
p_{k,N}\Bigl(\Bigl(\frac{N}{M_{k,N}}\Bigr)^{\sfrac{1}{2}}\beta\Bigr)
.
\end{align}
Note that \eqref{eq:grem:rem:restricted-free-energy} is the restricted free
energy (cf. \eqref{eq:grem:restricted-free-energy}) of the REM, where the
restriction is imposed by the total magnetization
\eqref{eq:grem:total-magnetisation} given by $t_{k,N}$.

We claim that the family of functions
$
\mathcal{P}
\equiv
\{ 
\E
\left[
p_N(\cdot)
\right]
\}_{N \in \N}
$ 
is uniformly Lipschitzian. Indeed, uniformly in $\beta \geq 0$, we have
\begin{align*}
\partial_{\beta}
\E
\left[
p_N(\beta)
\right]
=
N^{-1/2}
\E
\left[
\mathcal{G}_N(\beta,0)
\left[
X_N(\sigma)
\right]
\right]
\leq
N^{-1/2}
\E
\left[
\max_{
\sigma \in \Sigma_N
}
X(\sigma)
\right]
\xrightarrow[N \uparrow +\infty]{}
\sqrt{2 \log 2}
.
\end{align*}
Hence, the family 
$
\mathcal{P}
$ 
has uniformly bounded first derivatives.

Given $t \in
(-1;1)$ and $t_{k_N,N} \in U_N$ (cf. \eqref{eq:grem:overlap-range-abstract})
such that 
$
\lim_{
N \uparrow +\infty
} 
t_{k_N,N}
= 
t
$, using \eqref{eq:grem:binom-asymptotics}, we have
\begin{align}
\label{eq:grem:rem:restricted-volume-asymptotics}
\lim_{
N \uparrow +\infty
}
\frac{M_{k_N,N}}{N}
=
1-I(t)\log_2\ee
.
\end{align}
Using \eqref{eq:grem:rem:restricted-volume-asymptotics} and the uniform
Lipschitzianity of the family $\mathcal{P}$, we get
\begin{align}
\label{eq:grem:rem:restricted-limiting-free-energy}
\lim_{
N \uparrow +\infty
}
\widetilde{p}_{k_N, N}(\beta) 
=
(1-I(t)\log_2\ee)
p
\left(
\frac{\beta}{\sqrt{1-I(t)\log_2\ee}}
\right)
.
\end{align}
Combining \eqref{eq:grem:rem:restricted-limiting-free-energy} with
\eqref{eq:grem:global-free-energy-through-local-free-energy-abstract-ldp}
and \eqref{eq:grem:localisation-convergence-constant-variance}, we get
\begin{align}
\label{b-n-estimate}
p(\beta,h)
& 
=
\max_{t\in[-1;1]}
\left\{
t \beta h +
(1-I(t)\log_2\ee)
p
\left(
\frac{\beta}{\sqrt{1-I(t)\log_2\ee}}
\right)
\right\}
.
\end{align}
To find the maximum in \eqref{b-n-estimate}, we consider two cases. 
\begin{enumerate}
\item 
If 
$
\beta 
\leq 
\sqrt{2(\log 2 - I(t_{*}))}
$,
then according to \eqref{rem-limit}, we have
\begin{align*}
p
\left(
\frac{\beta}{\sqrt{1-I(t)\log_2\ee}}
\right)
=
\log 2
+
\frac{
\beta^2
}{
2 (1-I(t)\log_2\ee)
}
.
\end{align*}
Hence, \eqref{b-n-estimate} implies
\begin{align}
\label{eq:grem:rem:ext-field-high-temp}
p(\beta,h)
=
\max_{t\in[-1;1]}
\left\{
t \beta h 
+
\frac{\beta^2}{2}
+
\log 2
-
I(t)
\right\}
=
\log 2
+
\log \ch \beta h
+
\frac{\beta^2}{2} 
,
\end{align}
where the last equality is due to the fact that the expression in the curly
brackets is concave and, hence, the maximum is attained at a stationary
point. The stationarity condition reads
\begin{align}
\label{eq:grem:rem:high-temp-stationarity}
t = t_0(\beta,h) \equiv \tanh \beta h
.
\end{align}
It is easy to check that the following identity holds
\begin{align}
\label{eq:grem:rem:cramer-entropy-through-ath}
I(t) = t \tanh^{-1} t - \log \ch \tanh^{-1} t
.
\end{align}
Combining \eqref{eq:grem:rem:cramer-entropy-through-ath} and
\eqref{eq:grem:rem:high-temp-stationarity}, we get
\eqref{eq:grem:rem:ext-field-high-temp}.
\item
If 
$
\beta 
\geq 
\sqrt{2(\log 2 - I(t_{*}))}
$,
then again by \eqref{rem-limit}, we have
\begin{align*}
p
\left(
\frac{\beta}{\sqrt{1-I(t)\log_2\ee}}
\right)
=
\frac{
\beta \sqrt{2 \log 2}
}{
\sqrt{1-I(t)\log_2 \ee}
}
.
\end{align*}
Hence, \eqref{b-n-estimate} transforms to
\begin{align}
\label{eq:grem:rem:ext-field-low-temp}
p(\beta,h)
=
\max_{t\in[-1;1]}
\left\{
t \beta h 
+
\beta 
\sqrt{
2(\log 2 - I(t))
}
\right\}
=
\beta
\left(
\sqrt{2(\log 2 - I(t_{*}))}+ht_{*}
\right)
,
\end{align}
where the last equality is due to the concavity of the expression in the curly
brackets.
\end{enumerate}
Combining \eqref{eq:grem:rem:ext-field-high-temp} and
\eqref{eq:grem:rem:ext-field-low-temp}, we get
\eqref{eq:grem:rem:ext-field-free-energy}.
\end{proof}
\begin{remark}
We note that due to the continuity of the free energy as a function of
$\beta$, we have at the freezing temperature $\beta_0$
\begin{align}
\label{eq:grem:rem-continuity-condition}
t_0(\beta_0,h) = t_*(h)
.
\end{align}
\end{remark}
Theorem~\ref{thm:grem:rem-ext-field} suggests that the following holds.
\begin{theorem}
Under the assumptions of Theorem~\ref{thm:grem:rem-ext-field}, we have
\begin{align}
\label{eq:grem:rem:limiting-max}
\lim_{
N \uparrow +\infty
}
\frac{1}{\sqrt{N}}
\max_{\sigma \in \Sigma_N} X_N(h,\sigma)
=
\sqrt{2(\log 2 - I(t_{*}))}+ht_{*}
,
\quad
\text{almost surely and in $L^1$.}
\end{align}
\end{theorem}
\begin{proof}
We have
\begin{align}
\label{eq:grem:rem:max-upper-bound-1}
\frac{1}{\beta}
p_N(\beta)
\leq
\frac{1}{N}
\log 
\left(
N
\beta
\sqrt{N}
\max_{
\sigma \in \Sigma_N
}
X_N(h,\sigma)
\right)
=
\frac{
\log N
}{
\beta
N
}
+
\frac{1}{\sqrt{N}}
\max_{
\sigma \in \Sigma_N
}
X_N(h,\sigma)
.
\end{align}
In view of \eqref{eq:grem:rem:ext-field-free-energy}, relation
\eqref{eq:grem:rem:max-upper-bound-1} readily implies that
\begin{align}
\label{eq:grem:rem:max-upper-bound}
\sqrt{2(\log 2 - I(t_{*}))}+ht_{*}
\leq
\varliminf_{
N \uparrow +\infty
}
N^{-1/2}
\max_{
\sigma \in \Sigma_N
}
X_N(h,\sigma)
.
\end{align}
We also have
\begin{align*}
\frac{1}{\beta}
p_N(\beta)
\geq
\frac{1}{\sqrt{N}}
\max_{
\sigma \in \Sigma_N
}
X_N(h,\sigma)
\end{align*}
which combined again with \eqref{eq:grem:rem:ext-field-free-energy}
implies that
\begin{align}
\label{eq:grem:rem:max-lower-bound}
\sqrt{2(\log 2 - I(t_{*}))}+ht_{*}
\geq
\varlimsup_{
N \uparrow +\infty
}
N^{-1/2}
\max_{
\sigma \in \Sigma_N
}
X_N(h,\sigma)
.
\end{align}
Due to the standard concentration of Gaussian measure (e.g.,
\cite[(2.35)]{LedouxBook2001}) and the fact that
the free energy \eqref{eq:grem:free-energy} is Lipschitzian with the constant
$\beta\sqrt{N}$ as a function of $X_N(h, \cdot)$ with respect to the Euclidean topology, the bounds \eqref{eq:grem:rem:max-upper-bound} 
and \eqref{eq:grem:rem:max-lower-bound} combined with the
Borell-Cantelli lemma give the convergence
\eqref{eq:grem:rem:limiting-max}.
\end{proof}

\subsection{Fluctuations of the ground state}
In this subsection, we shall study the limiting distribution of the point
process generated by the properly rescaled process of the energy levels,
i.e. \eqref{eq:grem:pp-of-the-rescaled-grem-energies}.
\begin{proof}[Proof of Theorem~\ref{thm:grem:rem:ground-state-fluctuations}]
Let us denote
\begin{align}
\label{eq:grem:rem:scaled-point-process}
\pointproc{E}_N(h) 
\equiv 
\sum_{
\sigma\in\Sigma_N
}
\delta_{u^{-1}_{N,h}(X_N(h,\sigma))}
.
\end{align}
We treat $\pointproc{E}_N(h)$ as a random pure point measure on $\R$.
Given some test function $\varphi \in C_0^+(\R)$ (i.e., a non-negative function
with compact support), consider the Laplace transform of
\eqref{eq:grem:rem:scaled-point-process} corresponding to $\varphi$
\begin{align}
\label{eq:grem:rem:laplace-transform}
L_{\pointproc{E}_N(h)}(\varphi) 
&
\equiv 
\E 
\left[
\exp 
\left\{ 
-
\sum_{\sigma\in\Sigma_N}
\varphi
\left(
u^{-1}_{N,h}(X_N(h,\sigma))
\right) 
\right\} 
\right]
\nonumber
\\
&
= 
\prod_{k=0}^N
\bigg\{
\frac{1}{2\pi}
\int_\R 
\exp
\left(
-
\varphi
\left(
u^{-1}_{N,h}(x+\frac{h}{\sqrt{N}}(N-2k))
\right)
-
\frac{x^2}{2}
\right)
\dd x 
\bigg\}^{\binom{N}{k}}
.
\end{align}
Introduce the new integration variables 
$
y=u^{-1}_{N,h}(x+\frac{h}{\sqrt{N}}(N-2k))
$. 
We have
\begin{align}
\text{\eqref{eq:grem:rem:laplace-transform}}
&
= 
\prod_{k=0}^N 
\bigg\{
\frac{A_N(h)}{2\pi}
\int_\R 
\exp
\bigg(
-\varphi(y)-\frac{1}{2}\Bigl(u_{N,h}(y)-\frac{h}{\sqrt{N}}(N-2k)\Bigr)^2
\bigg)
\dd y 
\bigg\}^{
\binom{N}{k}
}
\nonumber
\\
\label{eq:grem:rem:integral}
&
= 
\exp
\bigg\{
\sum_{k=0}^N 
\binom{N}{k} 
\log
\bigg(
1-\frac{A_N(h)}{\sqrt{2\pi}}
\int_\R(1-\ee^{-\varphi(y)})
\exp 
\left[
-\frac{1}{2}\Bigl(u_{N,h}(y)-\frac{h}{\sqrt{N}}(N-2k)\Bigr)^2 
\right]
\bigg) 
\bigg\}
.
\end{align}
Note that the integration in \eqref{eq:grem:rem:integral} is actually
performed over $y \in \supp \varphi$, since the integrand is zero
on the complement of the support. It is easy to check that uniformly in $y \in
\supp \varphi$ the integrand in \eqref{eq:grem:rem:integral} and, hence, the
integral itself are exponentially small (as $N \uparrow +\infty$). Consequently,
we have
\begin{align}
\label{eq:grem:rem:integral-2}
\text{\eqref{eq:grem:rem:integral}}
\underset{N \uparrow +\infty}{=}
&
\exp
\bigg\{
-
\int_{\supp \varphi}
(1-\ee^{-\varphi(y)})
\nonumber
\\
&
\sum_{k=0}^N 
\binom{N}{k} 
\frac{A_N(h)}{\sqrt{2\pi}}
\exp 
\left[
-
\frac{1}{2}\Bigl(u_{N,h}(y)-\frac{h}{\sqrt{N}}(N-2k)\Bigr)^2 
\right]
\bigg) 
(1+o(1))
\bigg\}
.
\end{align}
Denote $t_{k,N} \equiv \frac{N-2k}{N}$.
Using Lemma~\ref{prp:grem:binom-asympt}, we get
\begin{align}
\label{eq:grem:rem:integral-3}
\text{\eqref{eq:grem:rem:integral-2}}
\underset{N \uparrow +\infty}{=}
&
\exp
\bigg\{
-
(1+o(1))
\int_{\supp \varphi}
(1-\ee^{-\varphi(y)})
\nonumber
\\
&
\times
\sum_{k=0}^N 
\frac{
A_N(h)
}{
\pi
\left(
N(1-t_{k,N}^2)
\right)^{1/2}
}
\exp
\left[
N
\left(
\log 2-I(t_{k,N})
\right)
-
\frac{1}{2}
\Bigl(
u_{N,h}(y)-h t_{k,N}\sqrt{N}
\Bigr)^2 
\right]
\bigg) 
\bigg\}
.
\end{align}
Note that despite the fact that Lemma~\ref{prp:grem:binom-asympt} is
valid only for $t_{k,N}  \in [-1+\eps; 1-\eps]$, we can still write
\eqref{eq:grem:rem:integral-3}, since the both following sums are negligible:
\begin{align*}
0 
\leq
&
\sum_{
k
:
t_{k,N} 
\in 
\left(
[-1;-1+\eps] \cup [1-\eps,1]
\right)
}
\frac{
A_N(h)
}{
\pi
\left(
N(1-t_{k,N}^2)
\right)^{1/2}
}
\exp
\Bigl[
N
\left(
\log 2-I(t_{k,N})
\right)
\\
&
\quad
-
\frac{1}{2}
\Bigl(
u_{N,h}(y)-h t_{k,N}\sqrt{N}
\Bigr)^2 
\Bigr]
\leq
K
N
\exp
\left(
- L N
\right)
,
\end{align*}
and
\begin{align*}
0 
\leq
\sum_{
k
:
t_{k,N} 
\in 
\left(
[-1;-1+\eps] \cup [1-\eps,1]
\right)
}
\binom{N}{k} 
\frac{A_{N}(h)}{\sqrt{2\pi}}
\exp
\left[
-
\frac{1}{2}
\Bigl(
u_{N,h}(y)-h t_{k,N}\sqrt{N}
\Bigr)^2 
\right]
\leq
K N
\exp
\left(
-
L N
\right)
.
\end{align*}
Consider the sum appearing in
\eqref{eq:grem:rem:integral-3}
\begin{align}
\label{eq:grem:rem:the-binomial-sum}
S_N(h,y)
\equiv
\sum_{k=0}^N 
\frac{
A_N(h)
}{
\pi
\left(
N(1-t_{k,N}^2)
\right)^{1/2}
}
\exp
\left[
N
\left(
\log 2-I(t_{k,N})
\right)
-
\frac{1}{2}
\Bigl(
u_{N,h}(y)-h t_{k,N}\sqrt{N}
\Bigr)^2 
\right]
.
\end{align}
Introduce the functions $f_N,g_N: [-1;1] \to \R$ as
\begin{align*}
f_N(t)
&
\equiv
I(t)
+
\frac{1}{2}
\Bigl(
\frac{
u_{N,h}(y)
}{
\sqrt{N}
}
-
h t
\Bigr)^2 
-
\log 2
,
\\
g_N(t)
&
\equiv
\frac{
A_N(h)
}{
\pi
\left(
N(1-t^2)
\right)^{1/2}
}
.
\end{align*}
Note that definition \eqref{eq:grem:rem:ground-state} implies
\begin{align}
\label{eq:grem:rem:mean-field-equation}
I^{\prime}(t_*)
=
h
\rho(t_*)
.
\end{align}
A straightforward computation using \eqref{eq:grem:rem:scaling-a}, 
\eqref{eq:grem:rem:scaling-b} and \eqref{eq:grem:rem:mean-field-equation} gives
\begin{align}
\label{eq:grem:rem:laplace-convex-function-2nd-derivative}
f_N^{\prime\prime}(t)
&
= 
I^{\prime\prime}(t)+h
>
0
,
\\
\label{eq:grem:rem:laplace-convex-function-1st-derivative}
f_N^{\prime}(t_*)
& 
=
-
\frac{h}{
\left(
2 \rho(t_*) N
\right)
}
\left[
2 y
+
\log
\left(
\frac{
I^{\prime\prime}(t_*)
+
h
}{
4 \pi (1-t_*^2) (\log 2-I(t_*)) N
}
\right)
\right]
=
O
\left(
\frac{\log N}{N}
\right)
,
\\
\label{eq:grem:rem:laplace-convex-function}
f_N(t_*)
&
=
-
\frac{1}{N}
\left[
y
+
\frac{1}{2}
\log
\left(
\frac{
I^{\prime\prime}(t_*)
+
h
}{
4 \pi (1-t_*^2) (\log 2-I(t_*)) N
}
\right)
\right]
+
o
\left(
\frac{1}{N}
\right)
.
\end{align}
Hence, since \eqref{eq:grem:rem:laplace-convex-function-1st-derivative} vanishes
even after being multiplied by $\sqrt{N}$,
\eqref{eq:grem:rem:laplace-convex-function-1st-derivative} is negligible for
the purposes of the asymptotic Laplace principle. This readily implies that
uniformly in $y \in \supp \varphi$
\begin{align}
\label{eq:grem:rem:laplace}
S_N(h,y)
\underset{N \uparrow +\infty}{\sim}
\frac{
N
g_N(t_*)
}{2}
\left(
\frac{
2 \pi f_N^{\prime\prime}(t_*)
}{
N
}
\right)^{1/2}
\exp
\left[
N f_N(t_*)
\right]
.
\end{align}
Using
\eqref{eq:grem:rem:laplace-convex-function-2nd-derivative},
\eqref{eq:grem:rem:laplace-convex-function-1st-derivative} and
\eqref{eq:grem:rem:laplace-convex-function} in the r.h.s. of
\eqref{eq:grem:rem:laplace}, we obtain that uniformly in $y \in \supp \varphi$
\begin{align}
\label{eq:grem:rem:s-n-final-equivalence}
S_N(h,y)
\underset{N \uparrow +\infty}{\sim}
\exp(-y)
.
\end{align}
Finally, combining \eqref{eq:grem:rem:s-n-final-equivalence} and
\eqref{eq:grem:rem:integral-3}, we obtain
\begin{align}
\label{eq:grem:laplace-transform-converges-to-the-ppp}
\lim_{
N \uparrow +\infty
}
L_{\pointproc{E}_N(h)}(\varphi) 
=
\exp
\left\{
-
\int_{\R}
\left(
1-\ee^{-\varphi(y)}
\right)
\ee^{-y}
\dd y
\right\}
.
\end{align}
The r.h.s. of \eqref{eq:grem:laplace-transform-converges-to-the-ppp} is the
Laplace transform of $
\ppp 
\left(
\ee^{-x} \dd x
,
x \in \R
\right)
$. 
Then a standard result implies the claim
\eqref{grem:rem:energies-weak-convergence}.
\end{proof}

\subsection{Fluctuations of the partition function}
In this subsection, we compute the weak limiting distribution of the partition
function under the natural scaling induced by
\eqref{eq:grem:rem:energy-scaling}.
Define
\begin{align}
C_{N}(\beta,h)
&
\equiv
\exp
\left[
\beta M(h) N
+
\frac{
\beta
}{
2 \rho(t_*)
}
\log
\left(
\frac{
I^{\prime\prime}(t_*)
+
h
}{
4 \pi (1-t_*^2) (\log 2-I(t_*)) N
}
\right)
\right]
,
\\
D_{N}(\beta,h)
&
\equiv
\ch^{-2/3} (\beta h)
\exp
\left[
N
\left(
\log 2
+
\log \ch \beta h
+
\frac{\beta^2}{2} 
\right)
\right]
,
\\
\alpha(\beta,h)
&
\equiv 
\frac{\beta}{
\rho(t_*)
}
.
\end{align}
\begin{theorem}
If $\beta > \rho(t_*)$, then
\begin{align}
\label{eq:grem:rem:fluctuations-high-temperature}
\frac{
Z_{N}(\beta,h) 
}{
C_{N}(\beta,h)
}
\xrightarrow[N\uparrow+\infty]{w} 
\int_\R 
\ee^{
\alpha(\beta,h) x
}
\dd\mathcal{P}^{(1)}(x)
.
\end{align}
If $\beta < \rho(t_*)$, then
\begin{align}
\label{eq:grem:rem:fluctuations-low-temperature}
\frac{
Z_{N}(\beta,h) 
}{
D_{N}(\beta,h)
}
\xrightarrow[N\uparrow+\infty]{w} 
1
.
\end{align}
\end{theorem}
\begin{proof}
This is a specialization of
Theorem~\ref{thm:grem:grem:partition-function-fluctuations} which is proved in
Section~\ref{sec:grem:fluctuations-of-the-partition-function}.
\end{proof}

\section{The GREM with external field}
\label{sec:grem:grem}
In this section, we obtain the main results of the paper concerning the GREM
with external field. We prove the limit theorems for the distribution of the
partition function and that of the ground state. As a simple consequence of
these fluctuation results, we obtain an explicit formula for the free energy of
the GREM with external field.

\subsection{Fluctuations of the ground state}
As in the REM, we start from the ground state fluctuations (cf.
Theorem~\ref{thm:grem:rem:ground-state-fluctuations}). The following is the
main technical result of this section that shows exactly in which
situations the GREM with external field has the same scaling limit behavior as
the REM with external field.
\begin{proposition}
\label{prp:grem:rem-regime-of-grem-with-ext-field}
Either of the following two cases holds
\begin{enumerate}
\item
\label{grem-case-1}
If, for all 
$
l \in [2;n] \cap \N
$,
\begin{align}
\label{grem-strict-ineq}
\frac{
\log 2 - I(t_*(\theta_{l,n}^{-1/2} h))
}{
\log 2 - I(t_*(h))
}
< 
\theta_{l,n}
,
\end{align}
then we have
\begin{align}
\label{eq:grem:strict-ineq-weak-convergence}
\sum_{
\sigma\in\Sigma_N
}
\delta_{u^{-1}_{N,h}(X_N(h,\sigma))}
\xrightarrow[N\to\infty]{w}
\ppp(\ee^{-x}, x \in \R^d)
.
\end{align}
\item
\label{grem-case-2}
If, for all 
$
l \in [2,\dots,n] \cap \N
$, 
\begin{align}
\label{grem-conditions}
\frac{
\log 2 - I(t_*(\theta_{l,n}^{-1/2} h))
}{
\log 2 - I(t_*(h))
}
\leq
\theta_{l,n}
,
\end{align}
and there exists (at least one) 
$
l_0 \in [2;n] \cap \N
$
\begin{align}
\label{eq:grem:equality-condition}
\frac{
\log 2 - I(t_*(\theta_{l_0,n}^{-1/2} h))
}{
\log 2 - I(t_*(h))
}
=
\theta_{l_0,n}
,
\end{align}
then there exits the constant $K=K(\varrho, h)\in(0;1)$ such that
\begin{align}
\label{eq:grem:non-strict-condition-weak-convergence}
\sum_{
\sigma\in\Sigma_N
}
\delta_{u^{-1}_{N,h}(X_N(h,\sigma))}
\xrightarrow[N\to\infty]{w}
\ppp(
K\ee^{-x}
,
x \in \R
)
.
\end{align}
\end{enumerate}
\end{proposition}
\begin{remark}
If condition (\ref{grem-conditions}) is
violated, i.e., there exists $l_0 \in [2;n] \cap \N$ such that
\begin{align}
\frac{
\log 2 - I(t_*(\theta_{l,n}^{-1/2} h))
}{
\log 2 - I(t_*(h))
}
>
\theta_{l,n}
,
\end{align}
then the REM scaling (cf. \eqref{grem:rem:energies-weak-convergence},
\eqref{eq:grem:rem:energy-scaling}) is too strong to reveal the structure of
the ground state fluctuations of the GREM.
Theorem~\ref{thm:grem:limiting-grem-point-process} shows how the scaling and
the limiting object should be modified to capture the fluctuations of the GREM in this regime.
\end{remark}
\begin{proof}
\begin{enumerate}
\item 
Denote $N_l \equiv \Delta x_l N$, for $l
\in [1;n]$.  We fix arbitrary test function $\varphi\in C_\text{K}^+(\R)$, i.e.,
a nonnegative function with compact support. Consider the Laplace transform
$L_{\pointproc{E}_N(h)}(\varphi)$ of the random measure $\pointproc{E}_N(h)$
evaluated on the test function $\varphi$.
\begin{align}
\label{eq:grem:grem-ext-field-laplace-transform}
L_{\pointproc{E}_N(h)}(\varphi) 
&
\equiv 
\E
\left[
\exp\Bigl(-\sum_{\sigma\in\Sigma_N}(\varphi\circ
u_{N,h}^{-1})(X_N(h,\sigma))\Bigr) \right]
\nonumber
\\
&
= 
\E 
\left[
\prod_{\sigma\in\Sigma_N}\exp\Bigl(-(\varphi\circ u_{N,h}^{-1})(X_N(h,\sigma))\Bigr)
\right]
.
\end{align}
Consider also the family of i.i.d standard Gaussian random variables
\begin{align*}
\{
X(\sigma^{(l)},\sigma^{(2)},\ldots,\sigma^{(n)})
\mid
l \in [1;n]\cap\N
,
\sigma^{(l)} 
\in 
\Sigma_{N_l}
,
\ldots
,
\sigma^{(n)} 
\in 
\Sigma_{N_n}
\}
.
\end{align*}
Given $l \in [1;n] \cap \N$ and $y \in \R$, define 
\begin{align}
\label{eq:grem:l-n-l-v-def}
L_N(l,v)
\equiv
\E
\Bigl[
&
\prod_{
\sigma^{(l)} 
\shortparallel
\ldots
\shortparallel
\sigma^{(n)}
\in 
\Sigma_{(1-x_{l-1})N}
}
\exp
\Bigl(
-
\varphi \circ u_{N,h}^{-1}
(
v
+
a_{l} X(\sigma^{(l)})
+
\nonumber
\\
&
\ldots
+
a_n X(\sigma^{(l)},\ldots,\sigma^{(n)})
+
h
(1-x_{l-1})
\sqrt{N}
m(\sigma^{(l)},\ldots,\sigma^{(n)})
)
\Bigr)
\Bigr]
.
\end{align}
We readily have
\begin{align}
\label{eq:grem:full-laplace-equals-recursive-laplace}
L_{\pointproc{E}_N(h)}(\varphi)
=
L_N(1,0)
.
\end{align}
Due to the tree-like structure of the GREM, for $l \in [1;n-1] \cap
\N$, we have the following recursion
\begin{align}
\label{eq:grem:l-n-l-v-recursion}
L_N(l,v)
=
\prod_{
\sigma^{(l)} \in \Sigma_{N_l}
}
\E
\left[
L_N(
l+1
,
v
+
a_{l}X
+
h
\Delta x_l
\sqrt{N}
m(\sigma^{(l)})
)
\right]
,
\end{align}
where $X$ is a standard Gaussian random variable.
Introduce the
following quantities
\begin{align*}
& 
Y_N(h, y, v, t, l)
\equiv 
u_{N,h}(y)
-
h
(1-x_{l-1})
\sqrt{N}
t
-
v
.
\end{align*}
We claim that, for any $l \in [1;n] \cap \N$, uniformly in $v \in \R$
satisfying
\begin{align*}
v 
\leq 
\sqrt{N}
\left(
M(h)
-
\delta
-
\left(
1-q_{l-1}
\right)
\rho(h)
-
h
(1-x_{l-1})
t_*(\theta_{l,n}^{-1/2} h)
\right)
,
\end{align*}
we have
\begin{align}
\label{eq:grem:w-n-asymptotic}
&
\log
L_N(l,v)
\underset{N \uparrow +\infty}{\sim}
-
\frac{A_N(h)}{\sqrt{2\pi(1-q_{l-1})} } 
\sum_{
k = 0
}^{
(1-x_{l-1}) N 
}
\biggl(
\binom{(1-x_{l-1})N}{k}
\nonumber
\\
&
\quad\quad 
\times
\int_\R 
\bigl(1-\ee^{-\varphi(y)}\bigr)
\exp
\Bigl[
-
\frac{1}{2(1-q_{l-1})}
Y_N(h, y, v, t_{k,(1-x_{l-1})N}, l)^2 
\Bigr]
\dd y
\biggr)
.
\end{align}
We shall prove \eqref{eq:grem:w-n-asymptotic} by a decreasing induction in $l$
starting from $l=n$. 
\item
The base of induction is a minor modification of the proof
of Theorem~\ref{thm:grem:rem:ground-state-fluctuations}. By the definition
\eqref{eq:grem:l-n-l-v-def} and independence, we have
\begin{align}
\label{eq:grem:first-step-w-n}
L_N(n,v)
=
\prod_{k=0}^{N_n}
\Bigl(
\E
\exp
\Bigl[
-
(
\varphi\circ u_{h,N}^{-1})
(
a_n X
+
h \Delta x_n\sqrt{N}t_{k,N_n}
+
v
)
\Bigr]\Bigr)
^{
\binom{
N_n
}{
k
}
}
.
\end{align}
For fixed $k\in [0; N_n] \cap \Z$,
\begin{align}
\label{eq:grem:first-integration}
&
\E
\left[
\exp
\left(
-
(\varphi\circ u_{N,h}^{-1})(
a_n X
+
h \Delta x_n\sqrt{N}t_{k,N}
+
v
)
\right)
\right]
\nonumber
\\
& 
\quad 
= 
(2\pi)^{-\frac{1}{2}}
\int_\R \dd x 
\exp
\Bigl[
-x^2/2
-(
\varphi\circ
u_{N,h}^{-1}
)
(
a_n x
+ 
h \Delta x_n\sqrt{N}t_{k,N}
+
v
)
\Bigr]
.
\end{align}
We introduce in \eqref{eq:grem:first-integration} the new integration variable
\begin{align}
\label{eq:grem:first-change-of-variables}
y
\equiv 
u_{N,h}^{-1}
\left(
a_n x
+
h\Delta x_n\sqrt{N}t_{k,N_n}
+
v
\right)
.
\end{align}
Using the change of variables \eqref{eq:grem:first-change-of-variables}, we get
that the r.h.s. of \eqref{eq:grem:first-integration} is equal to
\begin{align}
\label{eq:grem:first-step-integration-using-y}
\frac{
A_N(h)
}{
\sqrt{2\pi} 
a_n
}\int_\R \dd y
\exp
\left[
-
\frac{1}{2a_n^2}
Y_N(h, y, v, t_{k,N_n}, n)^2
-
\varphi(y)
\right]
.
\end{align}
Combining \eqref{eq:grem:first-step-w-n} and
\eqref{eq:grem:first-step-integration-using-y}, we get
\begin{align*}
L_N(n,v)
& 
=
\prod_{k=0}^{N_n}
\Bigl(
\frac{A_N(h)}{\sqrt{2\pi} a_n}
\int_\R \dd y 
\exp
\Bigl[
-\frac{1}{2
a_n^2}
Y_N(h, y, v, t_{k,N_n}, n)^2
-
\varphi(y)
\Bigr]
\Bigr)^{
\binom{N_n}{k}
}
\\
&
=
\prod_{k=0}^{N_n}
\Bigl(
1
-
\frac{A_N(h)}{\sqrt{2\pi} a_n}
\int_\R 
\dd y 
\big(1-\ee^{-\varphi(y)}\bigr)
\exp
\Bigl[
-\frac{1}{2a_n^2}
Y_N(h, y, v, t_{k,N_n}, n)^2
\Bigr]
\Bigr)^{\binom{N_n}{k}}
.
\end{align*}
Define
\begin{align*}
V_N(h, v, t, n)
\equiv
\frac{A_N(h)}{\sqrt{2\pi} a_n}
\int_\R 
\dd y 
\big(1-\ee^{-\varphi(y)}\bigr)
\exp
\Bigl[
-\frac{1}{2a_n^2}
Y_N(h, y, v, t, n)^2
\Bigr]
.
\end{align*}
Given any small enough $\delta > 0$, it straightforward to show that uniformly
in $v \in \R$ such that
\begin{align*}
v
\leq 
\sqrt{N}
\left(
M(h) - h \Delta x_n t_*(h \theta_{n-1,n}^{-1/2}) - \delta
\right)
,
\end{align*}
we have
\begin{align}
\label{eq:grem:l-N-n-v-asymptotic}
L_N(n,v)
\underset{N \uparrow +\infty}{=}
\prod_{k=0}^{N_n}
\biggl(
1
-
\binom{N_n}{k}
V_N(h, v, t_{k,N_n}, n)
\biggr)
\bigl(
1+\OO(\ee^{-CN})
\bigr)
.
\end{align}
Indeed, we have
\begin{align*}
\exp
\Bigl[
-
\frac{1}{2a_n^2}
Y_N(h, y, v, t_{k,N_n}, n)^2
\Bigr]
\leq 
\exp
\Bigl[
-
N_n \bigl(\log 2-I(t_n)\bigr)
\Bigr]
.
\end{align*}
Next, using the fact that $ (1-\ee^{-\varphi(\cdot)})\in C_K^+(\R)$, we get
for some $C>0$
\begin{align*}
\int_\R \dd y 
\bigl(1-\ee^{-\varphi(y)}\bigr)
\exp
\Bigl[
-\frac{1}{2a_n^2}
Y_N(h, y, v, t_{k,N_n}, n)^2
\Bigr]
\leq 
C\exp\Bigl[- N_n \bigl(\log 2-I(t_n)\bigr)\Bigr]
.
\end{align*}
Applying the elementary bounds
\begin{align}
\label{eq:grem:elementary-log-bounds}
x-x^2 \leq \log(1+x) \leq x,\text{ for }|x|<\frac{1}{2}
\end{align}
to
\begin{align*}
x
\equiv 
-
\frac{A_N(h)}{\sqrt{2\pi} a_n}
\int_\R \dd y 
\bigl(1-\ee^{-\varphi(y)}\bigr)
\exp
\Bigl[
-
\frac{1}{2a_n^2}
Y_N(h, y, v, t_{k,N_n}, n)^2
\Bigr]
,
\end{align*}
and using the fact that, due to \eqref{eq:grem:binom-asymptotics}, there exists
$C>0$ such that uniformly in $k \in [1;N_n] \cap \N$
\begin{align*}
x^2 
\leq 
\exp
\Bigl[
-
2 N_n 
\bigl(\log 2-I(t_n)\bigr)
\Bigr]
\binom{N_n}{k}
\leq 
\ee^{-CN}
,
\end{align*}
we get \eqref{eq:grem:l-N-n-v-asymptotic} and, consequently,
\eqref{eq:grem:w-n-asymptotic} holds for $l=n$.
\item
For simplicity of presentation, we prove only the induction step $l=n
\leadsto l = n-1$. Due to \eqref{eq:grem:l-n-l-v-recursion}, we have
\begin{align}
\label{eq:l-N-n-1-asymptotics}
L_N(n-1,v)
=
\prod_{k_{n-1}=0}^{N_{n-1}}
\E
\left[
L_N
(
n
,
v
+
a_{n-1}X
+
h
\Delta x_{n-1}
\sqrt{N}
t_{k_{n-1},N_{n-1}}
)
\right]^{
\binom{N_{n-1}}{k_{n-1}}
} 
.
\end{align}
Define
\begin{align*}
t(k_n,k_{n-1})
\equiv
\frac{1}{
1-x_{l-2}
}
\left(
\Delta x_n t_{k_n,N_n}
+
\Delta x_{n-1} t_{k_{n-1},N_{n-1}}
\right)
.
\end{align*}
Fix an arbitrary $\delta > 0$ and $\eps > 0$. Due to
\eqref{eq:grem:w-n-asymptotic} with $l=n$,  there exists some $C>0$,
such that uniformly for all $k_n, k_{n-1}$ with
\begin{align*}
t_{k_n, k_{n-1}} 
\in
\{
t \in [-1;1]
:
\vert
t_*(\theta_{n-1,n}^{-1/2} h)
-
t_{k_n, k_{n-1}} 
\vert
\leq
\eps
\}
,
\end{align*}
and uniformly for all $v, x \in \R$  satisfying
\begin{align}
\label{eq:grem:condition-before-expectation}
\Delta x_n 
(\log 2 - I(t_{k_n,N_n}))
\leq
\frac{1}{2 a_n^2}
\Bigl(
&
M(h)
-
\delta
-
a_{n-1} x
-
N^{-1/2} v 
\nonumber
\\
&
-
h
(
\Delta x_n t_{k_n,N_n}
+
\Delta x_{n-1} t_{k_{n-1},N_{n-1}}
)
\Bigr)^2
,
\end{align}
we obtain
\begin{align}
\label{eq:grem:log-l-n-vanishes}
\left\vert
\log
L_N
(
n
,
v
+
a_{n-1}x
+
h
\Delta x_{n-1}
\sqrt{N}
t_{k,N_{n-1}}
)
\right\vert
\leq
C N 
\exp (-N/C)
.
\end{align}
Define
\begin{align}
\label{eq:grem:x-n-v-definition}
x_N(v)
\equiv
\frac{\sqrt{N}}{a_{n-1}}
\Bigl(
&
M(h)
-
\delta
-
v N^{-1/2}
-
a_n (2 \Delta x_n (\log 2 - I(t_{k_n,N_n})))^{1/2}
\nonumber
\\
&
-
h
(
\Delta x_n t_{k_n,N_n}
+
\Delta x_{n-1} t_{k_{n-1},N_{n-1}}
)
\Bigr)
.
\end{align}
Using the elementary bounds
\begin{align}
\label{eq:grem:elementary-exp-bounds}
1+x \leq \ee^{x} \leq 1+x+x^2,\text{ for }|x|<1
,
\end{align}
and the bound \eqref{eq:grem:log-l-n-vanishes}, we obtain
\begin{align}
\label{eq:grem:expexct-conditioned-l-N}
&
\E
\left[
\I_{
\{
X \leq x_N(v)
\}
}
L_N
(
n
,
v
+
a_{n-1}X
+
h
\Delta x_{n-1}
\sqrt{N}
t_{k_{n-1},N_{n-1}}
)
\right]
\nonumber
\\
&
\underset{N \uparrow +\infty}{=}
\P
\{
X \leq x_N(v)
\}
+
\E
\left[
\I_{
\{
X \leq x_N(v)
\}
}
\log 
L_N
(
n
,
v
+
a_{n-1}X
+
h
\Delta x_{n-1}
\sqrt{N}
t_{k_{n-1},N_{n-1}}
)
\right]
\nonumber
\\
&
\quad\quad
+
\mathcal{O}(N \exp(-N / C))
.
\end{align}
Given $k_{n-1} \in [1;N_{n-1}] \cap \N $, we have
\begin{align}
\label{eq:grem:restricted-exp-y}
&
\E
\left[
\I_{
\{
X \leq x_N(v)
\}
}
\exp
\left(
-
\frac{1}{2 a_n^2}
Y_N(
h
, 
y
, 
v
+
a_{n-1}X
+
h
\Delta x_{n-1}
\sqrt{N}
t_{k_{n-1},N_{n-1}}
, 
t_{k_{n},N_{n}}
,
n
)^2 
\right)
\right]
\nonumber
\\
&
= 
\frac{1}{\sqrt{2\pi}}
\int_{-\infty}^{x_N(v)}
\dd x
\exp
\left[
-
\frac{x^2}{2}
-
\frac{1}{2 a_n^2}
\left(
u_{N,h}(y)
-
a_{n-1} x
-
h
\sqrt{N}
(
\Delta x_n t_{k_n,N_n}
+
\Delta x_{n-1} t_{k_{n-1},N_{n-1}}
)
-
v
\right)^2
\right]
\nonumber
\\
&
= 
\exp
\left(
-
\frac{1}{
1-q_{n-2}
}
Y_N(
h
, 
y
, 
v
, 
t(k_n,k_{n-1})
,
n-1
)^2 
\right)
\nonumber
\\
&
\quad
\times
\frac{1}{\sqrt{2\pi}}
\int_{-\infty}^{x_N(v)}
\exp
\left[
-\frac{a_n^2+a_{n-1}^2}{2a_n^2}
\left(
x
-
\frac{a_{n-1}}{a_n^2+a_{n-1}^2}
Y_N(
h
, 
y
, 
v
, 
t(k_n,k_{n-1})
,
n-1
)
\right)^2
\right]
\dd x
.
\end{align}
We claim that due to the strict inequalities \eqref{grem-strict-ineq}, we have
\begin{align}
\label{eq:grem:restricted-integral}
\frac{1}{\sqrt{2\pi}}
\int_{-\infty}^{x_N(v)}
&
\exp
\left[
-\frac{a_n^2+a_{n-1}^2}{2a_n^2}
\left(
x
-
\frac{a_{n-1}}{a_n^2+a_{n-1}^2}
Y_N(
h
, 
y
, 
v
, 
t(k_n,k_{n-1})
,
n-1
)
\right)^2
\right]
\dd x
\nonumber
\\
&
\xrightarrow[N \uparrow +\infty]{}
\frac{
a_n
}{
\left(
a_n^2 + a_{n-1}^2
\right)^{1/2}
}
,
\end{align}
uniformly in $v \in \R$
such that
\begin{align}
\label{eq:grem:v-upper-bound-for-n-1}
v 
\leq 
\sqrt{N}
\left(
M(h)
+
\delta^\prime
-
h
(
\Delta x_n t_{k_n,N_n}
+
\Delta x_{n-1} t_{k_{n-1},N_{n-1}}
)
-
(a_n^2 + a_{n-1}^2)
\rho(h)
\right)
\equiv
v^{\text{max}}_N
,
\end{align}
where
$0 < \delta^\prime$ exists due to strict inequality
\eqref{grem-strict-ineq}, for $l = n$. Indeed, due to the standard bounds on
Gaussian tails, to show \eqref{eq:grem:restricted-integral} it is enough to
check that
\begin{align}
\label{eq:grem:center-smaller-than-right-tail}
\frac{a_{n-1}}{a_n^2+a_{n-1}^2}
Y_N(
h
, 
y
, 
v
, 
t(k_n,k_{n-1})
,
n-1
)
+
\delta
\sqrt{N}
\leq
x_N(v)
,
\end{align}
for $v$ satisfying \eqref{eq:grem:v-upper-bound-for-n-1}. 
Due to \eqref{grem-strict-ineq} with $l=n$, there exists $\delta_3 > 0$ such
that we have
\begin{align}
\label{eq:grem:the-condition-for-l-equals-n}
(2 \Delta x_n (\log 2 - I(t_{k_n,N_n})))^{1/2}
\leq
\rho(h)-\delta_3
.
\end{align}
Choosing a
small enough $\delta^\prime > 0$, we have
\begin{align*}
&
x_N(v)
-
\frac{a_{n-1}}{a_n^2+a_{n-1}^2}
Y_N(
h
, 
y
, 
v
, 
t(k_n,k_{n-1})
,
n-1
)
+
\delta
\sqrt{N}
\\
&
=
a_n^2 
(
M(h)
-
v N^{-1/2} 
- 
h 
\left(
\Delta x_n t_{k_n,N_n}
+
\Delta x_{n-1} t_{k_{n-1},N_{n-1}}
)
\right)
\\
&
\quad
-
(a_n^2 + a_{n-1}^2)
\left(
a_n (2 \Delta x_n (\log 2 - I(t_{k_n,N_n})))^{1/2}
-
\delta
\right)
\\
&
\underset{\eqref{eq:grem:v-upper-bound-for-n-1}}{\geq}
a_n^2
\left(
(a_n^2+a_{n-1}^2)
\rho(h)
-
\delta^\prime
\right)
-
(a_n^2+a_{n-1}^2)
\left(
a_n (2 \Delta x_n (\log 2 - I(t_{k_n,N_n})))^{1/2}
-
\delta
\right)
\\
&
\underset{\eqref{eq:grem:the-condition-for-l-equals-n}}{\geq}
a_n^2
\left(
(a_n^2+a_{n-1}^2)
\rho(h)
-
\delta^\prime
\right)
-
(a_n^2+a_{n-1}^2)
\left(
a_n^2 \rho(h)
-
\delta
\right)
\\
&
=
(
a_n^2 + a_{n-1}^2
)
(
\delta_3
a_n^2
+
\delta
)
-
a_n^2
\delta^\prime
>
0
\end{align*}
which proves \eqref{eq:grem:center-smaller-than-right-tail}.

We claim that there exists $C>0$ such
that uniformly in $k_{n-1} \in [1;N_{n-1}] \cap \N$ and in $v \in \R$ satisfying
\eqref{eq:grem:v-upper-bound-for-n-1} we have
\begin{align}
\label{eq:grem:restriction-doesnt-matter}
\binom{N_{n-1}}{k_{n-1}}
\P
\{
X \geq x_N(v)
\}
\leq
\exp(-N / C)
.
\end{align}
Indeed, in view of \eqref{eq:grem:binom-asymptotics} and due to the classical
Gaussian tail asymptotics, to obtain \eqref{eq:grem:restriction-doesnt-matter} it is
enough to show that
\begin{align}
\label{eq:grem:suffiecient-condition-for-getting-rid-of-the-restriciton}
N_{n-1} (\log 2 - I(t_{k_{n-1},N_{n-1}}))
\leq
\frac{1}{2}
x_N^2(v^{\text{max}}_N)
.
\end{align}
Using \eqref{eq:grem:v-upper-bound-for-n-1} and
\eqref{eq:grem:x-n-v-definition}, we obtain
\begin{align}
\label{eq:grem:x-n-of-v-max}
x_N(v^{\text{max}}_N) 
=
\frac{N^{1/2}}{a_{n-1}}
\left(
(a_n^2 + a_{n-1}^2) \rho(h)
-
a_n
(2 \Delta x_n (\log 2 - I(t_{k_n,N_n})))^{1/2}
+
\delta^\prime - \delta
\right)
.
\end{align}
If $n > 2$, then due to strict inequality \eqref{grem-strict-ineq}, for $l =
n-2$, there exists $\delta^{\prime\prime}>0$ such that we have
\begin{align}
\label{eq:grem:estimate-slpian}
(a_n^2 + a_{n-2}^2)
\rho(h)
-
\delta^{\prime\prime}
&
>
\left(
(
\log 2 - I(t_*(\theta_{l,n}^{-1/2} h))
)
(
a_n^2 + a_{n-2}^2
)
(
\Delta x_n + \Delta x_{n-1}
)
\right)^{1/2}
\nonumber
\\
&
\geq
(
2 a_{n-1}^2 \Delta x_{n-1} (\log 2 - I(t_{k_{n-1},N_{n-1}}))
)^{1/2}
\nonumber
\\
&
\quad
+
(
2 a_{n}^2 \Delta x_{n} (\log 2 - I(t_{k_{n},N_{n}}))
)^{1/2}
,
\end{align}
where the last inequality may be obtained as a consequence of Slepian's
lemma \cite{Slepian1962}. If $n=2$, then \eqref{eq:grem:estimate-slpian} follows
directly from Slepian's lemma. Combining \eqref{eq:grem:x-n-of-v-max} and
\eqref{eq:grem:estimate-slpian}, we get
\eqref{eq:grem:suffiecient-condition-for-getting-rid-of-the-restriciton}. Note
that \eqref{eq:grem:restriction-doesnt-matter},
in particular, implies that
\begin{align}
\label{eq:grem:restriction-tails-vanish}
\P
\{
X \geq x_N(v)
\}
\leq
\exp(-N / C)
.
\end{align}
Given $k_{n-1} \in [1;N_{n-1}] \cap \N$, denote
\begin{align*}
L_N(n-1,v, k_{n-1})
\equiv
\E
\left[
L_N
(
n
,
v
+
a_{n-1}X
+
h
\Delta x_{n-1}
\sqrt{N}
t_{k_{n-1},N_{n-1}}
)
\right]
^{
\binom{N_{n-1}}{k_{n-1}}
}
\end{align*}
Due to \eqref{eq:grem:restriction-tails-vanish} and
\eqref{eq:grem:expexct-conditioned-l-N}, we have
\begin{align*}
L_N(n-1,v, k_{n-1})
&
=
\E
\left[
(
\I_{
\{
X \leq x_N(v)
\}
}
+
\I_{
\{
X > x_N(v)
\}
}
)
L_N
(
n
,
v
+
a_{n-1}X
+
h
\Delta x_{n-1}
\sqrt{N}
t_{k_{n-1},N_{n-1}}
)
\right]
^{
\binom{N_{n-1}}{k_{n-1}}
}
\\
&
=
\Bigl(
1
+
\E
\left[
\I_{
\{
X \leq x_N(v)
\}
}
L_N
(
n
,
v
+
a_{n-1}X
+
h
\Delta x_{n-1}
\sqrt{N}
t_{k_{n-1},N_{n-1}}
)
\right]
\\
&
\quad\quad
+
\mathcal{O}
\left(
\P
\{
X \geq x_N(v)
\}
+
N \exp(-N/C)
\right)
\Bigr)^{
\binom{N_{n-1}}{k_{n-1}}
}
.
\end{align*}
Using \eqref{eq:grem:restriction-doesnt-matter} and the standard bounds
\eqref{eq:grem:elementary-log-bounds} and \eqref{eq:grem:elementary-exp-bounds},
we get
\begin{align*}
L_N(n-1,v, k_{n-1})
=
\exp
\biggl\{
&
\binom{N_{n-1}}{k_{n-1}}
\E
\left[
\I_{
\{
X \leq x_N(v)
\}
}
\log 
L_N
(
n
,
v
+
a_{n-1}X
+
h
\Delta x_{n-1}
\sqrt{N}
t_{k_{n-1},N_{n-1}}
)
\right]
\\
&
+
\mathcal{O}
(
N \exp(-N/C)
)
\biggr\}
.
\end{align*}
Applying \eqref{eq:grem:restricted-integral},
\eqref{eq:grem:restricted-exp-y}, \eqref{eq:grem:w-n-asymptotic}, for $l=n$, we
obtain
\begin{align*}
\log
L_N(n-1,v, k_{n-1})
=
-
&
\frac{
A_N(h)
}{
\sqrt{
2\pi(a_n^2+a_{n-1}^2)
} 
} 
\sum_{
k_n = 0
}^{
N_n
}
\biggl(
\binom{N_n}{k_n}
\binom{N_{n-1}}{k_{n-1}}
\nonumber
\\
&
\times
\int_\R 
\bigl(1-\ee^{-\varphi(y)}\bigr)
\exp
\Bigl[
-
\frac{1}{2(a_n^2+a_{n-1}^2)}
Y_N(h, y, v, t_{k_n,k_{n-1}}, n-1)^2 
\Bigr]
\dd y
\biggr)
\\
&
+
\mathcal{O}
(
N \exp(-N/C)
)
.
\end{align*}
Finally, we arrive at
\begin{align*}
\log L_N(n-1,v)
&
=
\sum_{
k_{n-1}=0
}^{
N_{n-1}
}
\log L_N(n-1,v, k_{n-1})
\\
&
=
-
\frac{A_N(h)}{\sqrt{2\pi(a_n^2 + a_{n-1}^2)} } 
\sum_{
k_n = 0
}^{
N_n 
}
\sum_{
k_{n-1} = 0
}^{
N_{n-1}
}
\biggl(
\binom{N_n}{k_n}
\binom{N_{n-1}}{k_{n-1}}
\nonumber
\\
&
\quad\quad 
\times
\int_\R 
\bigl(1-\ee^{-\varphi(y)}\bigr)
\exp
\Bigl[
-
\frac{1}{2(a_n^2 + a_{n-1}^2)}
Y_N(h, y, v, t_{k_n,k_{n-1}}, n-1)^2 
\Bigr]
\dd y
\biggr)
\\
&
\quad
+
\mathcal{O}
(
N^2 \exp(-N/C)
)
\\
&
=
-
\frac{A_N(h)}{\sqrt{2\pi(a_n^2 + a_{n-1}^2)} } 
\sum_{
k = 0
}^{
N_n + N_{n-1} 
}
\biggl(
\binom{N_n+N_{n-1}}{k}
\nonumber
\\
&
\quad\quad 
\times
\int_\R 
\bigl(1-\ee^{-\varphi(y)}\bigr)
\exp
\Bigl[
-
\frac{1}{2(a_n^2 + a_{n-1}^2)}
Y_N(h, y, v, t_{k,N_n+N_{n-1}}, n-1)^2 
\Bigr]
\dd y
\biggr)
\\
&
\quad
+
\mathcal{O}
(
N^2 \exp(-N/C)
)
. 
\end{align*}
\item
Combining \eqref{eq:grem:full-laplace-equals-recursive-laplace} and
\eqref{eq:grem:w-n-asymptotic} for $l=1$, we obtain
\begin{align}
\label{grem-laplace-transform}
L_{\pointproc{E}_N(h)}(\varphi) 
& 
= 
\exp
\biggl(
-
\int_\R \bigl(1-\ee^{-\varphi(y)}\bigr)S_N(h,y)\dd y
+
o(1)
\biggl)
,
\end{align}
where $S_N(h,y)$ is given by \eqref{eq:grem:rem:the-binomial-sum}. Invoking the
proof of Theorem~\ref{thm:grem:rem:ground-state-fluctuations}, we get that
\begin{align*}
L_{\pointproc{E}_N(h)}(\varphi) 
& 
\xrightarrow[N \uparrow +\infty]{} 
\exp
\left(
-
\int_\R \bigl(1-\ee^{-\varphi(y)}\bigr)\ee^{-y} \dd y
\right)
\\
& 
= 
L_{\pointproc{P}(\ee^{-x})}(\varphi)
.
\end{align*}
This establishes \eqref{eq:grem:strict-ineq-weak-convergence}.
\item
The proof of \eqref{eq:grem:non-strict-condition-weak-convergence}  is very similar to the
above proof of \eqref{eq:grem:strict-ineq-weak-convergence}. The main
difference is that \eqref{eq:grem:restricted-integral} does not hold. Instead, 
if \eqref{eq:grem:equality-condition} holds for $l_0 = n$, then we have
\begin{align}
\label{eq:grem:critical-case-convergence}
&
\frac{1}{\sqrt{2\pi}}
\int_{-\infty}^{x_N(v)}
\exp
\left[
-\frac{a_n^2+a_{n-1}^2}{2a_n^2}
\left(
x
-
\frac{a_{n-1}}{a_n^2+a_{n-1}^2}
Y_N(
h
, 
y
, 
v
, 
t(k_n,k_{n-1})
,
n-1
)
\right)^2
\right]
\dd x
\nonumber
\\
&
\xrightarrow[N \uparrow +\infty]{}
\frac{
a_n
}{
\left(
a_n^2 + a_{n-1}^2
\right)^{1/2}
}
\P
\bigg\{
X
< 
\frac{
\sqrt{N}
}{
a_{n-1}
\sqrt{a_n^2 + a_{n-1}^2}
}
\Bigl[
M(h)
-
v N^{-1/2}
-
(1-x_{n-2})
h
t_*(h \theta_{n,n}^{-1/2})
\nonumber
\\
&
\quad\quad\quad
-
(a_n^2 + a_{n-1}^2)
\rho(h)
\Bigr]
\bigg\}
,
\end{align}
uniformly in 
\begin{align*}
v 
\leq 
\sqrt{N}
\left(
M(h)
-
(1-x_{n-2})
h
t_*(h \theta_{n,n}^{-1/2})
-
(a_n^2 + a_{n-1}^2)
\rho(h)
\right)
-
\delta^\prime
.
\end{align*}
The subsequent applications of the recursion \eqref{eq:grem:l-n-l-v-recursion}
to \eqref{eq:grem:critical-case-convergence} give rise to the
constant $K(h,\varrho) \in (0;1)$ in
\eqref{eq:grem:non-strict-condition-weak-convergence}.
\end{enumerate}
\end{proof}
\begin{proof}[Proof of Theorem~\ref{thm:grem:limiting-grem-point-process}]
The existence of the r.h.s. of \eqref{eq:grem:ground-state-fluctuations} follows
from \cite[Theorem~1.5 (ii)]{BovierKurkova2004}. It remains to show 
convergence~\eqref{eq:grem:ground-state-fluctuations} itself. We apply
Proposition~\ref{prp:grem:rem-regime-of-grem-with-ext-field} to each
coarse-grained block. Note that the assumption \eqref{grem-strict-ineq} of
Proposition~\ref{prp:grem:rem-regime-of-grem-with-ext-field} is fulfilled, due
to the construction of the blocks, cf. \eqref{eq:grem:coarse-graining-indeces},
\eqref{eq:grem:decreasing-slopes}. The result then follows from
\cite[Theorem~1.2]{BovierKurkova2004}.

The representation of the limiting ground state
\eqref{eq:grem:limiting-ground-state} is proved exactly as in \cite[Theorem~1.5 (iii))]{BovierKurkova2004}.
\end{proof}
\subsection{Fluctuations of the partition function} 
\label{sec:grem:fluctuations-of-the-partition-function}
In this subsection we compute the limiting distribution of the GREM partition
function under the scaling induced by
\eqref{eq:grem:rem:energy-scaling}. The analysis amounts to
handling both the low and high temperature regimes. The low
temperature regime is completely described by the behavior of the ground
states which is summarized in
Theorem~\ref{thm:grem:limiting-grem-point-process}. The high temperature
regime is considered in
Lemma~\ref{eq:grem:fluctuations-of-the-partition-function-l-equals-zero} below.
\begin{lemma}
\label{eq:grem:fluctuations-of-the-partition-function-l-equals-zero}
Assume $l(\beta,h) = 0$. Then
\begin{align}
\label{eq:grem:high-temperature-convergence}
\exp
&
\Biggl[
-
N
\left(
\log 2
+
\log \ch \beta h
+
\frac{\beta^2}{2} 
\right)
\Biggr]
\ch^{2/3} (\beta h)
Z_N(\beta,h)
\nonumber
\\
&
\xrightarrow[N\uparrow+\infty]{
w
}
K(\beta,h)
,
\end{align}
where $K(\beta,h) = 1$, if 
$
\beta \bar{\gamma}_1(h) < 1
$, 
and 
$K(\beta,h) \in (0;1)$, if 
$
\beta \bar{\gamma}_1(h) 
=
1
$.
\end{lemma}
\begin{proof}
We follow the strategy of \cite[Lemma~3.1]{BovierKurkova2004}. By the very
construction of the coarse-graining
algorithm \eqref{eq:grem:coarse-graining-indeces}, we have
\begin{align}
\label{eq:grem:high-temp-slopes-ordering}
\widetilde{\theta}_{1,k}
&
\leq
\widetilde{\theta}_{1,J_1}
=
\bar{\gamma}_1(h)^2
,
\quad
k \in [1;J_1] \cap \N
,
\nonumber
\\
\widetilde{\theta}_{1,k}
&
<
\widetilde{\theta}_{1,J_1}
,
\quad
k \in (J_1;n] \cap \N
.
\end{align}
Assume 
$
\beta \bar{\gamma}_1(h) < 1
$. 
Hence, due to
\eqref{eq:grem:high-temp-slopes-ordering}, we have
\begin{align}
\label{eq:grem:high-temp-strict-ineq-with-beta}
\beta \widetilde{\theta}_{1,k}^{1/2}
<
1
,
\quad
k \in [1;n] \cap \N
.
\end{align}
Strict inequality \eqref{eq:grem:high-temp-strict-ineq-with-beta} implies
that there exists $\eps > 0$ such that, for all $k \in [1;n] \cap \N$,
\begin{align}
\label{eq:grem:high-temp-strict-ineq-with-beta-with-eps}
\left(
\beta^2 - \frac{1}{2}(\beta - \eps)^2
\right)
q_k
<
x_k
\left(
\log 2 - I(t_*(h (x_k/q_k)^{1/2}))
\right)
.
\end{align}
We have
\begin{align}
\label{eq:grem:high-temperature-fluctuations-scaling}
\E
\left[
Z_N(\beta,h)
\right]
&
=
\sum_{k=0}^N
\binom{N}{k}
\exp
\left(
\beta h t_{k,N} N
+
\frac{\beta^2 N}{
2
}
\right)
\equiv
S_N(\beta,h)
.
\end{align}
Note that due to \eqref{eq:grem:binom-asymptotics}
\begin{align}
\label{eq:grem:high-temperature-s-n-binomial-asymptotics}
S_N(\beta,h)
\underset{N \uparrow +\infty}{\sim}
\sum_{k=0}^N
g_N(t_{k,N})
\exp
\left(
N
f(t_{k,N})
\right)
,
\end{align}
where
\begin{align*}
f(t)
&
\equiv
\log 2
-
I(t)
+
\beta h t
+
\beta^2/2
,
\\
g_N(t)
&
\equiv
\left(
\frac{2}{
\pi N (1-t^2)
}
\right)^{1/2}
.
\end{align*}
A straightforward computation gives
\begin{align*}
f^\prime(t_0) 
&
=
\beta h - \tanh^{-1}(t_0 (\beta,h))
=
0
\\
f^{\prime\prime}(t_0) 
&
=
-
(1-t_0^2)^{-1}
=
-
\ch^2(\beta h)
,
\\
g_N(t_0)
&
=
\left(
\frac{2}{
\pi N (1-t^2)
}
\right)^{1/2}
=
\left(
\frac{2}{
\pi N
}
\right)^{1/2}
\ch (\beta h)
.
\end{align*}
The asymptotic Laplace method then yields
\begin{align}
\label{eq:grem:high-temperature-fluctuations-scaling-asymptotics}
S_N(\beta,h)
\underset{N \uparrow +\infty}{\sim}
\ch^{-2/3} (\beta h)
\exp
\left[
N
\left(
\log 2
+
\log \ch \beta h 
+
\frac{\beta^2 }{2} 
\right)
\right]
.
\end{align}
For $p \leq q$, define 
\begin{align*}
\grem_N^{(p,q)}(\sigma^{(1)},\ldots,\sigma^{(q)})
\equiv
\sum_{k=p}^{q}
a_k
X(\sigma^{(1)},\ldots,\sigma^{(k)})
.
\end{align*}
Consider the event
\begin{align*}
E_N(\sigma)
\equiv
\Bigl\{
&
\grem_N^{(1,k)}(\sigma^{(1)},\ldots,\sigma^{(k)})
<
(\beta + \eps)
q_k
\sqrt{N}
,
\\
&
\text{for all }
k \in [1;n] \cap \N
\Bigr\}
.
\end{align*}
Define the truncated partition function as
\begin{align}
\label{eq:grem:high-temp-truncated-partition-function}
Z^{\text{(T)}}_N(\beta,h)
\equiv
\sum_{\sigma \in \Sigma_N}
\I_{
E_N(\sigma)
}
\exp
\left[
\beta
\sqrt{N}
X_N
\left(
h,\sigma
\right)
\right]
.
\end{align}
The truncation \eqref{eq:grem:high-temp-truncated-partition-function} is mild
enough in the following sense
\begin{align}
\label{eq:grem:high-temp-restircted-asymptotic}
\E
\left[
Z^{\text{(T)}}_N(\beta)
\right]
&
=
S_N(\beta,h)
\P
\left\{
\grem_N^{(1,k)}(\sigma^{(1)},\ldots,\sigma^{(k)})
<
\eps
q_k
\sqrt{N}
,
\text{for all }
k \in [1;n] \cap \N
\right\}
\nonumber
\\
&
\underset{N \uparrow +\infty}{\sim}
\E
\left[
Z_N(\beta,h)
\right]
.
\end{align}
We write
\begin{align*}
\frac{
Z_N(\beta)
}{
\E
\left[
Z_N(\beta)
\right]
}
&
=
\frac{
Z^{\text{(T)}}_N(\beta)
}{
\E
\left[
Z^{\text{(T)}}_N(\beta)
\right]
}
\times
\frac{
\E
\left[
Z^{\text{(T)}}_N(\beta)
\right]
}{
\E
\left[
Z_N(\beta)
\right]
}
+
\frac{
Z_N(\beta)
-
Z^{\text{(T)}}_N(\beta)
}{
\E
\left[
Z_N(\beta)
\right]
}
\\
&
\equiv
\text{(I)}
\times
\text{(II)}
+
\text{(III)}
.
\end{align*}
Due to \eqref{eq:grem:high-temp-restircted-asymptotic}, we get
\begin{align*}
\text{(II)}
\underset{N \uparrow +\infty}{\sim}
1,
\quad
\text{(III)}
\xrightarrow[N \uparrow +\infty]{L^1}
0
.
\end{align*}
To estimate $\text{(I)}$, we fix any $\delta > 0$, and use the Chebyshev
inequality
\begin{align}
\label{eq:grem:high-temp-fluctuations:chebyshev}
\P
\left\{
\vert
\text{(I)}-1
\vert
>
\delta
\right\}
&
\leq
\left(
\delta 
\E
\left[
Z^{\text{(T)}}_N(\beta)
\right]
\right)^{-2}
\var
\left[
Z^{\text{(T)}}_N(\beta)
\right]
.
\end{align}
Expanding the squares, we have
\begin{align}
\label{eq:grem:high-temp-variance-1}
\var
\left[
Z^{\text{(T)}}_N(\beta)
\right]
&
=
\E
\left[
Z^{\text{(T)}}_N(\beta)^2
\right]
-
\E
\left[
Z^{\text{(T)}}_N(\beta)
\right]^2
\nonumber
\\
&
=
\sum_{p=1}^{n}
\sum_{
\sigma^{(1)}
\shortparallel
\ldots
\shortparallel
\sigma^{(k)}
\in
\Sigma_{
x_k N
}
}
\E
\Bigg[
\exp
\Bigl\{
2
\beta
\sqrt{N}
\Bigl(
\grem_N^{(1,p)}(\sigma^{(1)},\ldots,\sigma^{(p)})
\nonumber
\\
&
\quad\quad
+
2 \beta h x_p m_{x_p N}(\sigma^{(1)},\ldots,\sigma^{(p)}) \sqrt{N}
\Bigr)
\Bigr\}
\nonumber
\\
&
\quad
\times
\sum_{
\substack{
\sigma^{(p+1)}
\shortparallel
\ldots
\shortparallel
\sigma^{(n)}
,
\\
\tau^{(p+1)}
\shortparallel
\ldots
\shortparallel
\tau^{(n)}
\in 
\Sigma_{(1-x_p)N}
,
\nonumber
\\
\substack{
\sigma^{(p+1)}
\neq
\tau^{(p+1)}
}
}
}
\exp
\Bigl\{
\beta
\sqrt{N}
\nonumber
\\
&
\quad\quad
\times
\Bigl(
\grem_N^{(p+1,n)}(\sigma^{(1)},\ldots,\sigma^{(n)})
+
\grem_N^{(p+1,n)}(\tau^{(1)},\ldots,\tau^{(n)})
\nonumber
\\
&
\quad\quad
+
h
(1-x_{p})
\sqrt{N}
(
m_{(1-x_{p})N}
(
\sigma^{(p+1)}
\shortparallel
\ldots
\shortparallel
\sigma^{(n)}
)
+
m_{(1-x_{p})N}
(
\tau^{(p+1)}
\shortparallel
\ldots
\shortparallel
\tau^{(n)}
)
\Bigr)
\Bigr\}
\nonumber
\\
&
\quad
\times
\I_{
E_N(
\sigma^{(1)}
\shortparallel
\ldots
\shortparallel
\sigma^{(n)}
)
}
\I_{
E_N(
\tau^{(1)}
\shortparallel
\ldots
\shortparallel
\tau^{(n)}
)
}
\Bigg]
.
\end{align}
Hence, due to the independence, we arrive at
\begin{align}
\label{eq:grem:high-temp-variance-2}
\var
\left[
Z^{\text{(T)}}_N(\beta)
\right]
&
\leq
\sum_{p=1}^{n}
\sum_{k=0}^{
x_k N
}
\binom{N}{k}
\E
\Bigg[
\exp
\Bigl\{
2
\beta
\sqrt{N}
\Bigl(
\grem_N^{(1,p)}(\sigma^{(1)},\ldots,\sigma^{(p)})
\nonumber
\\
&
\quad\quad
+
h x_p t_{k,N} \sqrt{N}
\Bigr)
\Bigr\}
\I_{
\left\{
\grem_N^{(1,p)}(\sigma^{(1)},\ldots,\sigma^{(p)})
<
(\beta + \eps)
q_p
\sqrt{N}
\right\}
}
\Bigg]
\nonumber
\\
&
\quad
\times
\Biggl(
\sum_{k=0}^{(1-x_p)N}
\binom{(1-x_p)N}{k}
\E
\Bigg[
\exp
\Bigl(
\beta
\sqrt{N}
(
\grem_N^{(p+1,n)}(\sigma^{(1)},\ldots,\sigma^{(n)})
\nonumber
\\
&
\quad\quad
+
h
(1-x_{p})
t_{k,(1-x_{p}) N}
\sqrt{N}
\Bigr)
\Bigg]
\Biggl)^2
.
\end{align}
Assume that $X$ is a standard Gaussian random variable. Using the
standard Gaussian tail bounds, we have
\begin{align}
\label{eq:grem:high-temp:numerator:asymptotics}
&
\E
\Bigg[
\exp
\Bigl(
2
\beta
\sqrt{N}
(
\grem_N^{(1,p)}(\sigma^{(1)},\ldots,\sigma^{(p)})
+
h x_p t_{k,N} \sqrt{N}
)
\Bigr)
\nonumber
\\
&
\quad\quad
\I_{
\left\{
\grem_N^{(1,p)}(\sigma^{(1)},\ldots,\sigma^{(p)})
<
(\beta + \eps)
q_p
\sqrt{N}
\right\}
}
\Bigg]
\nonumber
\\
&
=
\exp
\left\{
N
\left(
2 \beta^2 q_p + \beta h  t_{k,N}
\right)
\right\}
\P
\left\{
X
\geq
(\beta - \eps)
\sqrt{
q_p
N
}
\right\}
\nonumber
\\
&
\underset{N \uparrow +\infty}{\leq}
C
\exp
\left\{
N
\left(
2 \beta^2 q_p + \beta h  t_{k,N}
-
\frac{1}{2}
(\beta-\eps)^2
q_p
\right)
\right\}
.
\end{align}
Similarly to \eqref{eq:grem:high-temperature-s-n-binomial-asymptotics}, using
\eqref{eq:grem:binom-asymptotics} and
\eqref{eq:grem:high-temp:numerator:asymptotics}, we have
\begin{align}
\label{eq:grem:high-temp-numerator-p-part}
&
\sum_{k=0}^{
x_k N
}
\binom{N}{k}
\E
\Bigg[
\exp
\Bigl(
2
\beta
\sqrt{N}
(
\grem_N^{(1,p)}(\sigma^{(1)},\ldots,\sigma^{(p)})
\nonumber
\\
&
\quad\quad
+
h x_p t_{k,N} \sqrt{N}
)
\Bigr)
\I_{
\left\{
\grem_N^{(1,p)}(\sigma^{(1)},\ldots,\sigma^{(p)})
<
(\beta + \eps)
q_p
\sqrt{N}
\right\}
}
\Bigg]
\nonumber
\\
&
\underset{N \uparrow +\infty}{\leq}
C
\sum_{k=0}^{
x_k N
}
\exp
\left\{
N
\left(
x_p (\log 2 - I(t_{k,x_p N}))
+
2 \beta^2 q_p
+
2 \beta h x_p t_{k,x_p N}
-
\frac{1}{2}
(\beta-\eps)^2
q_p
\right)
\right\}
\equiv
P_N(p)
.
\end{align}
Using \eqref{eq:grem:binom-asymptotics}, we also obtain
\begin{align}
\label{eq:grem:high-temp-numerator-p-tilde-part}
&
\sum_{k=0}^{(1-x_p)N}
\binom{(1-x_p)N}{k}
\E
\Bigg[
\exp
\Bigl(
\beta
\sqrt{N}
(
\grem_N^{(p+1,n)}(\sigma^{(1)},\ldots,\sigma^{(n)})
\nonumber
\\
&
\quad\quad
+
h
(1-x_{p})
\sqrt{N}
t_{k,(1-x_{p}) N}
\Bigr)
\Bigg]
\nonumber
\\
&
\underset{N \uparrow +\infty}{\leq}
C
\sum_{k=0}^{(1-x_p) N}
\exp
\Bigl\{
N
\Bigl(
(1-x_p) (\log 2 - I(t_{k,(1-x_p) N}))
+
\frac{1}{2}
(1-q_p) \beta^2 
\nonumber
\\
&
\quad\quad\quad\quad
+
\beta h (1-x_p) t_{k, (1-x_p) N}
\Bigr)
\Bigr\}
\equiv
\widetilde{P}_N(p)
.
\end{align}
Combining \eqref{eq:grem:high-temp-variance-2},
\eqref{eq:grem:high-temp-numerator-p-part} and
\eqref{eq:grem:high-temp-numerator-p-tilde-part}, we get
\begin{align}
\label{eq:grem:variance-upper-bound-3}
\var
\left[
Z^{\text{(T)}}_N(\beta)
\right]
\underset{N \uparrow +\infty}{\leq}
\sum_{p=1}^N
P_N(p)
\widetilde{P}_N^2(p)
.
\end{align}
For any $p \in [1;n] \cap \N$, we have the following factorization
\begin{align}
\label{eq:grem:high-temp-factorisation}
\E
\left[
Z^{\text{(T)}}_N(\beta)
\right]
&
=
\sum_{k=0}^{x_p N}
\binom{x_p N}{k}
\E
\Biggl[
\exp
\left(
\beta
\sqrt{N}
(
\grem_N^{(1,p)}(\sigma^{(1)},\ldots,\sigma^{(p)})
+
h x_p t_{k,x_p N} \sqrt{N}
)
\right)
\nonumber
\\
&
\quad
\times
\sum_{k=0}^{(1-x_p) N}
\binom{(1-x_p) N}{k}
\exp
\Bigl(
\beta
\sqrt{N}
(
\grem_N^{(p+1,n)}(\sigma^{(1)},\ldots,\sigma^{(n)})
\nonumber
\\
&
\quad\quad
+
h (1-x_p) N t_{k,(1-x_p) N} \sqrt{N}
)
\Bigr)
\I_{
E_N(
\sigma^{(1)}
\shortparallel
\ldots
\shortparallel
\sigma^{(n)}
)
}
\Biggr]
.
\end{align}
Hence, again similarly to
\eqref{eq:grem:high-temperature-s-n-binomial-asymptotics}, we obtain
\begin{align}
\label{eq:grem:high-temp-denomenator}
\E
\left[
Z^{\text{(T)}}_N(\beta)
\right]
&
\underset{N \uparrow +\infty}{\sim}
C
\sum_{k=0}^{x_p N}
\exp
\left\{
N
\left(
x_p (\log 2 - I(t_{k,x_p N}))
+
\frac{1}{2}
q_p \beta^2 
+
\beta h x_p t_{k,x_p N}
\right)
\right\}
\nonumber
\\
&
\quad\quad
\times
\sum_{k=0}^{(1-x_p) N}
\exp
\Bigl\{
N
\Bigl(
(1-x_p) (\log 2 - I(t_{k,(1-x_p) N}))
+
\frac{1}{2}
(1-q_p) \beta^2 
\nonumber
\\
&
\quad\quad\quad\quad
+
\beta h (1-x_p) t_{k, (1-x_p) N}
\Bigr)
\Bigr\}
\equiv
Q_{N}(p) 
\times
\widetilde{P}_N(p)
.
\end{align}
Denote
\begin{align*}
R_N(p)
\equiv
q_p \beta^2 
+
2
x_p
\max_{
t \in [-1;1]
}
\left\{
\log 2 - I(t)
+
\beta h  t
\right\}
.
\end{align*}
We observe that similarly to
\eqref{eq:grem:high-temperature-fluctuations-scaling-asymptotics} we have
\begin{align}
\label{eq:grem:high-temp-q-n-asymptotics}
\frac{
Q_{N}^2(p)
}{
\exp( N R(p))
}
\underset{N \uparrow +\infty}{\sim}
C
.
\end{align}
Combining \eqref{eq:grem:variance-upper-bound-3},
\eqref{eq:grem:high-temp-denomenator},
\eqref{eq:grem:high-temp-q-n-asymptotics} and
\eqref{eq:grem:rem-continuity-condition}, we get
\begin{align}
\label{eq:grem:high-temp-chebyshev-final-bound}
\eqref{eq:grem:high-temp-fluctuations:chebyshev}
&
\underset{N \uparrow +\infty}{\leq}
C
\sum_{p=1}^n
\frac{
P_{N}(p)
}{
Q_{N}^2(p)
}
=
C
\sum_{p=1}^n
\frac{
P_{N}(p) / \exp( N R(p))
}{
Q_{N}^2(p) / \exp( N R(p))
}
\underset{N \uparrow +\infty}{\leq}
C
\sum_{p=1}^n
\frac{
P_{N}(p)
}{
\exp( N R(p))
}
\nonumber
\\
&
\underset{N \uparrow +\infty}{\leq}
C
\sum_{p=1}^n
\exp
\left\{
N
\left(
(
\beta^2
-
\frac{1}{2}
(\beta - \eps)^2
)
q_p
-
(\log 2 - I(t_0)) x_p
\right)
\right\}
\nonumber
\\
&
=
C
\sum_{p=1}^n
\exp
\left\{
N
\left(
(
\beta^2
-
\frac{1}{2}
(\beta - \eps)^2
)
q_p
-
(
\log 2 - I(t_*(h (x_p/q_p)^{1/2}))
) x_p
\right)
\right\}
\xrightarrow[N \uparrow +\infty]{}
0
,
\end{align}
where the convergence to zero in the last line is assured by the choice of
$\eps$ in \eqref{eq:grem:high-temp-strict-ineq-with-beta-with-eps}. Finally, combining
\eqref{eq:grem:high-temp-fluctuations:chebyshev} and \eqref{eq:grem:high-temp-chebyshev-final-bound}, we get
\begin{align*}
\text{(I)}
\xrightarrow[N \uparrow +\infty]{\P}
1
.
\end{align*}
This finishes the proof of \eqref{eq:grem:high-temperature-convergence} in the
case $
\beta \bar{\gamma}_1(h)
<
1
$.

The case 
$
\beta \bar{\gamma}_1(h)
=
1
$
is a little bit more tedious and uses the information about the low
temperature regime obtained in
Theorem~\ref{thm:grem:limiting-grem-point-process} in the spirit of the
proof of \cite[Lemma~3.1]{BovierKurkova2004}. The lemma follows.
\end{proof}
\begin{proof}[Proof of
Theorem~\ref{thm:grem:grem:partition-function-fluctuations}]
The proof is verbatim the one of \cite[Theorem~1.7]{BovierKurkova2004}, where
the analysis of the high temperature regime \cite[Lemma~3.1]{BovierKurkova2004}
is substituted by
Lemma~\ref{eq:grem:fluctuations-of-the-partition-function-l-equals-zero}. The
low temperature regime is governed by the fluctuations of the ground state
which are summarized in Theorem~\ref{thm:grem:limiting-grem-point-process}.

\end{proof}

\subsection{Formula for the free energy of the GREM} 
\begin{proof}[Proof of Theorem~\ref{thm:grem:grem-free-energy}]
The $L^1$ convergence follows immediately from
Theorem~\ref{thm:grem:grem:partition-function-fluctuations}. Almost sure
convergence is a standard consequence of Gaussian measure concentration,
e.g., \cite[(2.35)]{LedouxBook2001}, and the Borell-Cantelli lemma.
\end{proof}

\section*{Acknowledgments}
A.K. gratefully acknowledges financial support of the DFG Research
Training Group ``Stochastic Processes and Probabilistic Analysis'' and of the
Helmholz-Gemeinschaft.

\bibliographystyle{plain}
\bibliography{bibliography}

\end{document}